\documentclass[12pt, reqno]{amsart}

\usepackage{amsrefs}
\usepackage{amssymb}
\usepackage{amsfonts}
\usepackage{amsmath}
\usepackage{mathrsfs}
\usepackage{comment}
\usepackage{hyperref}

\newcommand{\abs}[1]{|#1|}
\newcommand{\C}{\ensuremath{\mathbb{C}^n}}

\newcommand{\R}{\ensuremath{\mathbb{R}}}
\newcommand{\Z}{\ensuremath{\mathbb{Z}}}
\renewcommand{\S}{\ensuremath{\text{Sig}_d}}
\newcommand{\diag}{\ensuremath{\text{diag}}}
\newcommand{\Rd}{\ensuremath{{\mathbb{R}^d}}}
\newcommand{\D}{\ensuremath{\mathscr{D}}}

\newcommand{\J}{\ensuremath{\mathscr{J}}}
\newcommand{\F}{\ensuremath{\mathscr{F}}}

\newcommand{\BMOWW}{\ensuremath{{\text{BMO}}_{W, W}^p}}
\newcommand{\BMOWUDt}{\ensuremath{{\text{BMO}}_{W, U, \D^t} ^p}}
\newcommand{\BMOWUAltDt}{\ensuremath{\widetilde{{\text{BMO}}_{W, U, \D^t} ^p}}}

\newcommand{\BMOWUD}{\ensuremath{{\text{BMO}}_{W, U, \D} ^p}}
\newcommand{\BMOWUAltD}{\ensuremath{\widetilde{{\text{BMO}}_{W, U, \D} ^p}}}
\newcommand{\BMOWUtwoD}{\ensuremath{{\text{BMO}}_{W, U, \D} ^2}}

\newcommand{\BMOWU}{\ensuremath{{\text{BMO}}_{W, U} ^p}}
\newcommand{\BMOWUAlt}{\ensuremath{\widetilde{{\text{BMO}}_{W, U} ^p}}}
\newcommand{\BMOWUtwo}{\ensuremath{{\text{BMO}}_{W, U} ^2}}

\newcommand{\Mn}{\ensuremath{\mathcal{M}_{n }}(\mathbb{C})}

\newcommand{\inrd}{\ensuremath{\int_{\Rd}}}
\newcommand{\tr}{\ensuremath{\text{tr}}}
\newcommand{\ip}[2]{\ensuremath{\left\langle#1,#2\right\rangle}}
\newcommand{\BMO}{\ensuremath{\text{BMO}}}
\newcommand{\Atwo}[1]{\ensuremath{\|#1 \|_{\text{A}_2} }}
\newcommand{\V}[1]{\ensuremath{\vec{#1}}}
\newcommand{\MC}[1]{\ensuremath{\mathcal{#1}}}

\newcommand{\W}[1]{\ensuremath{\widetilde{#1}}}

\numberwithin{equation}{section}

\newtheorem{thm}{Theorem}[section]
\newtheorem{lemma}[thm]{Lemma}
\newtheorem{corollary}[thm]{Corollary}

\newtheorem{proposition}[thm]{Proposition}
\newtheorem*{prop*}{Proposition}

\theoremstyle{remark}

\newtheorem*{rem*}{Remark}

\begin{document}

\title[Two  weight inequalities for commutators]{Boundedness of commutators and H${}^1$-BMO duality in the two matrix weighted setting}

\author[J. Isralowitz]{Joshua Isralowitz}

\address{%
University at Albany\\
Mathematics Department\\
1400 Washington ave. \\
Albany, NY 12222}

\email{jisralowitz@albany.edu}


\begin{abstract}
In this paper we characterize the two matrix weighted boundedness of commutators with any of the Riesz transforms (when both are matrix A${}_p$ weights) in terms of a natural two matrix weighted BMO space.  Furthermore, we identify this BMO space when $p = 2$ as the dual of a natural two matrix weighted H${}^1$ space, and use our commutator result to provide a converse to Bloom's matrix A${}_2$ theorem, which as a very special case proves Buckley's summation condition for matrix A${}_2$ weights. Finally, we use our results to prove a matrix weighted John-Nirenberg inequality, and we also briefly discuss the challenging question of extending our results to the matrix weighted vector BMO setting.
\end{abstract}

\maketitle

\section{Introduction}

Let $w$ be positive a.e. on $\Rd$ and let $L^p(w)$ be the standard weighted Lebesgue space with respect to the norm \begin{equation*} \|f\|_{L^p( w)} = \left(\inrd |f(x)|^p w(x) \, dx \right)^\frac{1}{p}. \end{equation*} Furthermore, let A${}_p$ be the classical Muckenhoupt class of weights $w$ satisfying \begin{equation*} \sup_{\substack{I \subseteq \Rd \\ I \text{ is a cube}}} \left(\frac{1}{|I|} \int_I w (x) \, dx \right)\left(\frac{1}{|I|} \int_I w ^{-\frac{1}{p - 1}} (x) \, dx\right)^{p-1} < \infty. \end{equation*}

In the interesting paper \cite{B}, the author proved that if $w, u \in \text{A}_p$ then a locally integrable $b : \R \rightarrow \mathbb{C}$ satisfies $[H, b] : L^p(u) \rightarrow L^p(w)$ boundedly (where $H$ is the Hilbert transform) if and only if $b \in \text{BMO}_\nu$ where $\nu = (u w^{-1})^\frac{1}{p}$ and $b \in \text{BMO}_\nu$ if \begin{equation*} \sup_{\substack{I \subseteq \R \\ I \text{is an interval}} } \frac{1}{\nu(I)} \int_I |b(x) - b_I | \, dx  < \infty \end{equation*}  While this is well known and not surprising when $u = w$, in general this result is quite remarkable given that this characterization involves the three functions $u, v$ and $b$.

Note that Bloom's two weight characterization above was largely motivated by the question of when the Hilbert transform $H$ is bounded on matrix weighted $L^2$.  In particular, let $W : \Rd \rightarrow \Mn$ be  a matrix weight, i.e. a positive definite a.e. $\Mn$ valued function on $\Rd$ and let $L^p(W)$ be the space of $\C$ valued functions $\V{f}$ such that $\|\V{f}\|_{L^p(W)} < \infty$, where \begin{equation*} \|\V{f}\|_{L^p(W)}  = \left(\inrd |W^\frac{1}{p}(x) \V{f}(x)|^p \, dx \right)^\frac{1}{p}. \end{equation*}   Furthermore, we will say that a matrix weight $W$ is a matrix A${}_p$ weight (see \cite{R}) if it satisfies \begin{equation} \label{MatrixApDef} \sup_{\substack{I \subset \R^d \\ I \text{ is a cube}}} \frac{1}{|I|} \int_I \left( \frac{1}{|I|} \int_I \|W^{\frac{1}{p}} (x) W^{- \frac{1}{p}} (t) \|^{p'} \, dt \right)^\frac{p}{p'} \, dx  < \infty. \end{equation}

Now if $1 < p < \infty$, then it was shown in the late 1990's by the independent efforts of M. Goldberg, F. Nazarov/S. Treil, and A. Volberg (see \cite{G,NT,V}) that a CZO on $\Rd$ is bounded on $L^p(W)$ if $W$ is a matrix A${}_p$ weight.  Over a decade earlier, however, S. Bloom showed (using his two weight characterization above) in \cite{B} that if $W = U^* \Lambda U$ where $U : \R \rightarrow \Mn$ is unitary and $\Lambda = \diag (\lambda_1, \ldots, \lambda_n)$ where each $\lambda_k$ is a scalar A${}_2$ weight, then $H$ is bounded on $L^2(W)$ if for each $r$ and $j$ we have \begin{equation*} u_{rj} \in \text{BMO}_{(\lambda_r \lambda_k ^{-1})^\frac12} \text{ for } k = 1, 2, \ldots, n,\end{equation*} which given the results in \cite{G,NT,V} retranslates into a sufficient condition for $W  = U^* \Lambda U$ to be a matrix A${}_2$ weight given that $U$ is unitary and $\Lambda$ is diagonal with scalar A${}_2$ entries.

On the other hand, in the very recent preprint \cite{HLW}, the authors extended the results in \cite{B} to all CZOs on $\Rd$ for $d > 1$. Given these results, it is natural to try to prove two matrix weighted norm inequalities for commutators $[T, B]$ where $T$ is a CZO and $B$ is a locally integrable $\Mn$ valued function.  Moreover, it is natural to use these two matrix weighted norm inequalities to try to find improvements and generalizations to Bloom's matrix A${}_2$ theorem above.

The general purpose of this paper is to investigate these matters.  Before we state our results, let us rewrite Bloom's BMO condition in a way that naturally extends to the matrix weighted setting.  First, by multiple uses of the A${}_p$ property and H\"{o}lder's inequality, it is easy to see that \begin{align*} m_I \nu  & \approx (m_I u) ^\frac{1}{p} (m_I w^{1-p'})^\frac{1}{p'} \\ &  \approx (m_I u) ^\frac{1}{p} (m_I w)^{-\frac{1}{p}} \\ & \approx (m_I u^\frac{1}{p}) (m_I w^\frac{1}{p})^{-1} \end{align*} so that $b \in \BMO_\nu$ when $w$ and $u$ are A${}_p$ weights if and only if

\begin{equation*} \sup_{\substack{I \subseteq \R \\ I \text{ is a cube}} } \frac{1}{|I|} \int_I |(m_I w^\frac{1}{p}) (m_I u^\frac{1}{p})^{-1} | |b(x) - b_I | \, dx  < \infty \end{equation*}  Now if $W, U$ are matrix A${}_p$ weights, then we define $\BMOWU$ to be the space of $n \times n$ locally integrable matrix functions $B$  where there exists $\epsilon  > 0$ such that \begin{equation*} \sup_{\substack{I \subseteq \Rd \\ I \text{is a cube}} } \frac{1}{|I|} \int_I \|(m_I W^\frac{1}{p}) (B(x) - B_I)(m_I U^\frac{1}{p})^{-1} \|^{1 + \epsilon} \, dx  < \infty. \end{equation*}  We can now state the main result of the paper \begin{thm} \label{CommutatorThm} Let $W$ and $U$ be $\Mn$ valued matrix A${}_p$ weights on $\Rd$ and let $B$ be a locally integral $\Mn$ valued function.  If $R$ is any of the Riesz transforms then $[R, B]$ maps $L^p(U)$ to $L^p(W)$ boundedly if and only if $B \in \BMOWU$. \end{thm} \noindent Note that sufficiency in Theorem \ref{CommutatorThm} is new even in the scalar setting in the sense that \cite{HLW} proves that $b \in \BMO_\nu$  if \textit{all} of the Riesz transforms $R_j$ for $j = 1, \ldots, d$ are bounded from $L^p(u)$ to $L^p(w)$ when $u, w$ are scalar A${}_p$ weights.

Unfortunately we are at the moment not able to use Theorem \ref{CommutatorThm} to prove any kind of improvement to Bloom's matrix A${}_2$ theorem.      Intriguingly, however, we can easily use Theorem \ref{CommutatorThm} to prove a matrix A${}_p$ converse under vastly more general conditions. More precisely we will prove the following. \begin{thm} \label{MatrixWeightThm} Let $\Lambda$ be a matrix A${}_p$ weight and let $U$ be \textit{any} matrix function such that $W = (U^*\Lambda ^{\frac{2}{p}} U)^\frac{p}{2}$ is a matrix A${}_p$ weight.  Then  $U \in \text{BMO}_{\Lambda, W}^p $. \end{thm}

A curious application (and one that warrants further investigation into various generalizations of Theorem \ref{CommutatorThm} and Theorem \ref{MatrixWeightThm}) is when $W$ is a (given) matrix A${}_2$ weight, $p = 2, \ \Lambda = W^{-1}, \ $ and $U = W$, which in this case says $W \in \text{BMO}_{W^{-1}, W}$. As we will see later, this translates into a matrix Fefferman-Kenig-Pipher and Buckley condition on matrix A${}_2$ weights $W$.  Note that while the former is well known in the matrix setting (see \cite{TV, CW}), the latter is to the author's knowledge new.

Also, one can ask whether sufficiency in Theorem \ref{CommutatorThm} holds for general CZOs.  Before we discuss this we will need to introduce some notation.  Following the notation in \cite{LPPW}, for any dyadic grid in $\mathbb{R}$ and any interval in this grid, let  \begin{equation*} h_I ^1 = |I|^{-\frac{1}{2}} \chi_I (x), \,  \,  \,  \, \, \, \,  h_I ^0 (x) = |I|^{-\frac{1}{2}} (\chi_{I_\ell} (x) - \chi_{I_r} (x))  \end{equation*} where $I_\ell$ and $I_r$ are the left and right halves of $I$, respectively.   Now given any dyadic grid $\D$ in $\mathbb{R}^d,$  any cube $I = I_1 \times \cdots \times I_d$, and any $\varepsilon \in \{0, 1\}^{d}$, let $h_I ^\varepsilon = \Pi_{i = 1}^d h_{I_i} ^\varepsilon$.  It is then easily seen that $\{h_I ^\varepsilon: I \in \D, \  \varepsilon \in \S\}$  where $\S = \{0, 1\}^d \backslash \{\vec{1}\}$ is an orthonormal basis for $L^2(\Rd)$. Note that we will say $h_I ^\varepsilon$ is ``cancellative" if $\varepsilon \neq \vec{1}$ since in this case $\int_I h_I^\varepsilon = 0$.

For a dyadic grid $\D$ let $\text{BMO}_{\nu, \D}$ be the canonical dyadic version of $\text{BMO}_{\nu}.$ In the scalar weighted setting, sufficiency in Theorem \ref{CommutatorThm} for general CZOs was proved in \cite{HLW} using the known (see \cite{LLL}) duality   BMO${}_{\nu, \D} = (H^1 _\D(\nu))^*$ under the standard $L^2$ pairing (which is not needed in the Riesz transform case). Here,  $H^1 _\D(\nu)$ is the space of all $f$ such that $S_{\D} f \in L^1(\nu)$   where $S_{\D}$ is the dyadic square function defined by \begin{equation*}S_{\D} f (x)= \left( \sum_{\varepsilon \in \S} \sum_{I \in \D} \frac{|f_I ^\varepsilon|^2}{|I|}  \chi_I (x) \right)^\frac12. \end{equation*}

 For a matrix weight $W$ and an $\Mn$ valued function $\Phi$, let $S_{W ,\D}$ be the weighted square function defined by \begin{equation*} S_{W, \D} \Phi (x)= \left(\sum_{\varepsilon \in \S} \sum_{I \in \D} \frac{\|W^{\frac{1}{2}} (x) \Phi_I ^\varepsilon \|^2  }{|I|} \chi_I (x) \right)^\frac12 \end{equation*} and for another matrix weight $U$, let $M_U$ denote the Haar multiplier \begin{equation*} M_U \Phi = \sum_{\varepsilon \in \S} \sum_{I \in \D} \Phi_I ^\varepsilon (m_I U)^\frac12 h_I ^\varepsilon. \end{equation*}  Finally let $H^1_{W, U, \D}$ be the space of locally integral $n \times n$ matrix functions defined by \begin{equation*} H^1 _{W, U} = \{\Phi : \Rd \rightarrow \Mn \text{ s.t. }  S_{W^{-1}, \D} M_U \Phi \in L^1 \}\end{equation*} and for any $n \times n$ matricies $A, B$ let $\ip{A}{B}_{\tr}$ be the canonical Frobenius inner product defined by \begin{equation*} \ip{A}{B}_{\tr} = \tr AB^*. \end{equation*} By modifying the ideas in \cite{BW,LLL} we will prove the following matrix weighted duality result in the last section, which we hope to use (possibly in modified form) in a future paper to prove sufficiency in Theorem \ref{CommutatorThm} for general CZOs when $p = 2$.

\begin{thm} \label{H1BMOThm} $\BMOWUtwoD = (H^1 _{W, U, \D})^*$ under the canonical pairing $B(\Phi) := \ip{\Phi}{B}_{L^2}$ where the inner product is with respect to $\ip{}{}_{\tr}$ on $\Mn$.  \end{thm}

\noindent Note that it would be very interesting to try to prove a similar duality result when $p \neq 2$ (which would likely be useful in extending Theorem \ref{CommutatorThm} to general CZOs when $p \neq 2$.) Also note that unlike \cite{LLL}, which proves their duality result by using the standard idea of analyzing the weighted measure of certain level sets, we are forced to instead largely base our arguments on unweighted estimates, since the ``matrix weighted measure" of level sets (or any set for that matter) makes absolutely no sense.

 It should be noted that various different\ equivalent versions of BMO$_{\nu}$ were needed in \cite{HLW} to prove their main result.  Similarly, we will require a number of different equivalent versions of $\BMOWU$ throughout the paper.  Surprisingly, many of these various versions in the matrix weighted setting have already appeared in \cite{IKP, I2} in the special cases where either $U = W$ or when one of the matrix weights $W$ or $U$ is the identity.

 In the last section we will also use our results to prove the following matrix weighted John-Nirenberg inequality, which, except for the $\epsilon$,  extends the classical scalar weighted John-Nirenberg inequality in \cite{MW} when $p = 2$. \begin{proposition} \label{p=2JN} Let $W$ be a matrix A${}_2$ weight. Then there exists $\epsilon > 0$ such that \begin{equation*} \sup_{\substack{I \subseteq \Rd \\ I \text{ is a cube } }} \frac{1}{|I|} \int_I \|(m_I W)^{-\frac12} (B(x) - m_I B) (m_I W)^{-\frac12}\| ^{1 + \epsilon} \, dx < \infty \end{equation*}  iff \begin{equation*} \sup_{\substack{ I \subseteq \Rd \\ I \text{ is a cube }} } \frac{1}{|I|} \int_I \|W^{-\frac12 }(x) (B^*(x) - m_I B^*) (m_I W)^{-\frac12}\| ^2     \, dx < \infty. \end{equation*} \end{proposition}

  Now while commutators with respect to $\C$ valued functions do not make sense, one can ask what a natural $\BMOWU$ condition for $\C$ valued functions is and whether conditions similar to the ones discussed in this paper are equivalent for $\C$ valued functions.  Unfortunately, due to a lack of symmetry and duality, this appears to be a challenging question and we refer the reader to \cite{I2} where these matters are briefly discussed in the special case when $U$ is the identity.  Despite this, we should comment that the following result (which should be thought of as a ``weak" matrix analogue of Theorem $5$ in \cite{MW}) is true and will be proved in the last section.   \begin{proposition} \label{VectorJN} Let $W$ is be matrix A${}_{p}$ weight for $1 < p < \infty$ and let $\V{f} : \Rd \rightarrow \C$. Then  $\V{f} \in \text{BMO}$ iff  \begin{equation} \label{JNVecCond} \sup_{\substack{ I \subseteq \Rd \\ I \text{ is a cube }} }  \frac{1}{|I|} \int_I |W^\frac{1}{p} (x)   (V_I (W))^{-1} (\V{f}(x) - m_I \V{f}) | ^{p} \, dx < \infty.  \end{equation} In fact, if $W$ is a matrix A${}_{p, \infty}$  weight then $\V{f} \in \text{BMO}$ implies  \eqref{JNVecCond}.   \end{proposition} \noindent Note that we will not define the matrix A${}_{p, \infty}$ condition since we will only need well known properties about matrix A${}_{p, \infty}$ weights to prove this result.  Furthermore, it is interesting to ask whether  \eqref{JNVecCond} implies $\V{f} \in \text{BMO}$  when $W$ is a matrix A${}_{p, \infty}$ weight.

 Let us briefly comment on the ideas and techniques used in this paper.  Like \cite{HLW0, HLW}, the ideas and techniques in this paper are ``dyadic" in nature and are very different than the more classical ideas and techniques in \cite{B}.  However, since the techniques in \cite{HLW0, HLW} are obviously scalar weighted techniques, we will \textit{not} draw from them in this paper, but instead heavily rely on the ideas developed in two recent preprints, the first being \cite{IKP} by the author, H. K. Kwon, and Sandra Pott, and the second being \cite{I2} by the author.  It should be commented, however, that we will in fact use the papers \cite{HLW0, HLW} as a sort of ``guiding light" for recasting the various matrix weighted BMO conditions in \cite{IKP, I2} into two matrix weighted conditions. Also note that with this in mind, one can think of this paper as a kind of ``two weight unification" of some of the ideas and results in \cite{IKP, I2}.

 Finally, the careful reader will notice that despite its elegant appearance (from the matrix weighted $p \neq 2$ perspective), we will not actually have a need for the definition of $\BMOWU$ given and instead will work with various other equivalent definitions and show that these all coincide with $\BMOWU$.  The reader, however, should not be tricked into thinking that the original definition of $\BMOWU$ is nice looking but useless.  In fact, the original definition of $\BMOWW$ is a very natural and important BMO condition (after using Lemma \ref{RedOp-AveLem} twice) to consider when formulating and proving a $T1$ theorem regarding the boundedness of matrix kernelled CZOs on $L^p(W)$ when $W$ is a matrix A${}_p$ weight for $1 < p < \infty$ (see \cite{I2} for a precise statement and proof of such a $T1$ theorem). Interestingly note that \textit{unlike} in the scalar setting, $\BMOWW$ does not reduce to the classical unweighted John-Nirenberg space BMO (see \cite{IKP} for more information.)

 We will end the introduction by noting that despite the paper's length, it is largely self contained, and in particular we do not assume that the reader is necessarily familiar with the ideas or arguments in \cite{IKP, I2}. Furthermore, as in \cite{HLW}, we will not attempt to track the A${}_p$ dependence on any of our results with the exception of A${}_2$ dependence in Theorem \ref{H1BMOThm} (which we hope to use to prove quantitative weighted norm inequalities for commutators $[T, b]$ on $L^2(W)$ for a matrix A${}_2$ weight $W$, a scalar kernelled CZO $T$,  and a scalar function $b \in \text{BMO}$ in a forthcoming paper) and in our matrix weighted Buckley summation condition (see Proposition \ref{BuckleyProp}.)

\section{Two weight characterization of paraproducts }
As in \cite{IKP, HLW0, HLW}, we will prove sufficiency in  Theorem \ref{CommutatorThm} by proving two matrix weighted norm inequalities for dyadic paraproducts in terms of equivalent BMO conditions similar to the ones in \cite{HLW0, HLW} (and when $p = 2$ in particular prove a two matrix weighted generalization of Theorem $3.1$ in \cite{HLW0}.)   Given a matrix weight $W$, let $V_I(W, p)$ and $V_I' (W, p) $ be reducing operators satisfying $|I|^{- \frac{1}{p}} \|\chi_I W^\frac{1}{p}  \vec{e}\|_{L^p} \approx |V_I(W, p) \vec{e}|$ and $|I|^{- \frac{1}{p'}} \|\chi_I W^{-\frac{1}{p}} \vec{e}\|_{L^{p'}} \approx |V_I '(W, p)  \vec{e}|$ for any $\vec{e} \in \C$ (see \cite{G}). We will drop the $p$ dependence and simply write $V_I(W)$ instead of $V_I(W, p)$ and similarly for $V_I'(W)$.  This should not cause any confusion (and if it might we will revert to the original notation) since we will exclusively deal with matrix A${}_p$ weights.

In general these reducing operators are not unique, and note that the specific ones chosen are not important for any of our theorems.  However, note that when $p  = 2$ we may take $V_I (W, 2)$ to be the average $(m_I W)^\frac12$ and $V_I ' (W, 2)$ can be taken to be the average $(m_I W)^{-\frac12}$.  In general though, it is important to realise that these reducing operators for $p \neq 2$ are \textit{not} averages.  Despite this,  it is nonetheless very useful to think of them as appropriate averages of $W$, which is further justified by the following simple but important result proved in \cite{TV} when $p = 2$ and proved in \cite{IKP}  for general $1 < p < \infty$.   \begin{lemma}\label{RedOp-AveLem}  If $W$ is a matrix A${}_p$ weight then \begin{equation*} |  V_I '  (W) \vec{e} | \approx  |m_I (W^{-\frac{1}{p}}) \vec{e}  | \end{equation*} for any $\vec{e} \in \C$.   In particular, \begin{equation*} |m_I (W^{-\frac{1}{p}}) \vec{e}| \leq |V_I ' (W) \vec{e} | \leq \|W\|_{\text{A}_p} ^\frac{n}{p}  |m_I (W^{-\frac{1}{p}}) \vec{e}|.  \end{equation*}  \end{lemma}  \noindent Of course, applying this to the dual weight $U^{1-p'}$ when $U$ is a matrix A${}_p$ weight gives us that \begin{equation*} |m_I (U^{\frac{1}{p}}) \vec{e}| \approx |V_I (U)\vec{e} |. \end{equation*}

Now given a locally integrable function $B : \Rd \rightarrow \Mn$, define the dyadic paraproduct $\pi_B$ with respect to a dyadic grid $\D$ by \begin{equation} \label{ParaprodDef} \pi_B \vec{f} = \sum_{\varepsilon \in \S} \sum_{I \in \D} B_I ^\varepsilon (m_I \vec{f}) h_I ^\varepsilon \end{equation}
 where $B_I ^\varepsilon$ is the matrix of Haar coefficients of the entries of $B$ with respect to $I$ and $\varepsilon$,  and $m_I \vec{f}$ is the vector of averages of the entries of $\vec{f}$.

We will now describe some important tools that are needed to prove two matrix weighted norm inequalities for dyadic paraproducts and that will also be used throughout the paper.  First is the ``matrix weighted Triebel-Lizorkin imbedding theorem" from \cite{NT, V} in the $d = 1$ setting, and from \cite{I1} when $d > 1$, which says that if $W$ is a matrix A${}_p$ weight then \begin{equation} \label{LpEmbedding}  \|\vec{f}\|_{L^p(W)} ^p  \approx \inrd \left(\sum_{I \in \D} \sum_{\varepsilon \in \S}  \frac{| V_I(W) \vec{f}_I ^\varepsilon|^2}{|I|} \chi_I(x) \right)^\frac{p}{2} \, dx  \end{equation} where $\vec{f}_I ^\varepsilon$ is the vector of Haar coefficients of the components of $\vec{f}$.

Thanks to \eqref{LpEmbedding}, we have that \begin{equation*} \|\pi_B (U ^{-\frac{1}{p}} \vec{f})\|_{L^p(W)} ^p \approx \inrd \left(\sum_{\varepsilon \in \S} \sum_{I \in \D} \frac{| V_I(W) B_I ^\varepsilon m_I (U^{-\frac{1}{p}} \vec{f})  |^2}{|I|} \chi_I(x) \right)^\frac{p}{2} \, dx \end{equation*} which crucially allows us to reduce the two matrix weighted boundedness of $\pi_B$ to that of a sort of ``matrix weighted Carleson embedding theorem" which is much simpler to handle and can in fact be handled like it is in the matrix one weighted setting in \cite{IKP} (and will be stated momentarily).

To do this, we will need a modification of the stopping time from \cite{I1, IKP}, which can be thought of as a  matrix weighted adaption of the stopping time from \cite{KP, Pott}.  Now assume that $W$ is a matrix A${}_p$ weight and that $\lambda$ is large enough. For any cube $I \in \D$, let $\J(I)$ be the collection of maximal $J \in \D(I)$ such that either of the two conditions \begin{equation*}      \|V_J (W) V_I(W) ^{-1}\|  > \lambda \    \text{   or   }  \  \|V_J(W) ^{-1} V_I(W)\|  > \lambda,  \end{equation*}  or either of the two conditions  \begin{equation}      \|V_J (U) (V_I  (U)) ^{-1}\|  > \lambda \    \text{   or   }  \  \|V_I  (U) (V_J  (U)) ^{-1}\|  > \lambda.   \end{equation}
  Also, let $\F(I)$ be the collection of dyadic subcubes of $I$ not contained in any cube $J \in \J(I)$, so that clearly $J \in \F(J)$ for any $J \in \D(I)$.

Let $\J^0 (I) := \{I\}$ and inductively define $\J^j(I)$ and $\F^j(I)$ for $j \geq 1$ by $\J^j (I) := \bigcup_{J \in \J^{j - 1} (I)} \J(J)$ and $\F^j (I) := \bigcup_{J \in \J^{j - 1} (I)} \F(J)$. Clearly the cubes in $\J^j(I)$ for $j > 0$ are pairwise disjoint.  Furthermore, since $J \in \F(J)$ for any $J \in \D(I)$, we have that $\D(I) = \bigcup_{j = 1}^\infty \F^j(I)$.  We will slightly abuse notation and write $\bigcup \J(I)$ for the set $\bigcup_{J \in \J(I)} J$ and write $|\bigcup \J(I)|$ for $|\bigcup_{J \in \J(I)} J|$. By easy arguments in  \cite{I1}, we can pick $\lambda$   so that $|\bigcup \J ^j (I)| \leq 2^{-j} |I|$ for every $I \in \D$.

We can now state and prove the main result of this section (which of course characterizes the boundedness of $\pi_B : L^p(U) \rightarrow L^p(W)$ in terms of the matrix Haar coefficient sequence $\{B_I ^\varepsilon\}$.) Note that a similar one matrix weighted result was stated and proved in \cite{IKP}, and in particular we will heavily utilize the ideas from \cite{IKP} to prove the following result.

   \begin{thm} \label{CarEmbedThm} Let $1 < p < \infty$ and for a sequence $\{A_I ^\varepsilon\}$ of $n \times n$ matricies  let $\MC{B}(W, U, A, p)$ be defined by \begin{equation*} \MC{B}(W, U, A, p)= \sup_{K \in \D} \frac{1}{|K|} \sum_{I \in \D(K)} \sum_{\varepsilon \in \S} \|V_I(W) A_I ^\varepsilon (V_K (U)) ^{-1}\|^2. \end{equation*}  If  $W$ is a matrix A${}_p$ weight then the following are equivalent:
\begin{itemize}
 \item[(a)]  The operator $\Pi_A ^{W, U, p}$ defined by \begin{equation*} \Pi_A ^{W, U, p} \vec{f} := \sum_{\varepsilon \in \S} \sum_{I \in \D} V_I(W) A_I ^\varepsilon  m_I ( U^{ -\frac{1}{p}} \vec{f}) h_I ^\varepsilon \end{equation*} is bounded on $L^p(\Rd;\C)$
\item[(b)]  \begin{equation*} \sup_{J \in \D} \frac{1}{|J|} \sum_{\varepsilon \in \S} \sum_{I \in \D(J)} \|V_I(W) A_I ^\varepsilon (V_I  (U))   ^{-1}  \|^2 < \infty \end{equation*}
\item[(c)] $\MC{B}(W, U, A, p) < \infty$ if $2 \leq p < \infty$, and  $B(U^{1-p'}, W^{1-p'}, A^*, p') < \infty$ if  $1 < p \leq 2$.
\item[(d)] $\Pi_{A^*}  ^{U^{1-p'}, W^{1-p'}, p'}$ is bounded on $L^{p'}(\Rd;\C)$.
\end{itemize}
Furthermore, either of the conditions $\MC{B}(W, U, A, p) < \infty$  or  $\MC{B}(U^{1-p'}, W^{1-p'}, A^*, p') < \infty$ (for any $1 < p < \infty$) implies that (b) (or equivalently (a) or (d)) is true. \end{thm}

Before we prove this result, note that for a matrix function $B$ we will write $\Pi_B ^{W, U, p}$ when the sequence of matricies is the Haar coefficients of $B$. Also, while we will not need it, note that elementary linear algebra arguments give us that $\MC{B}(W, U, A, p) < \infty$ if and only if there exists $C$ independent of $K$ where \begin{equation} \sum_{\varepsilon \in \S} \sum_{I \in \D(K)} (A_I ^\varepsilon)^* (V_I(W))^2 A_I ^\varepsilon \leq C (V_K(U))^2 \label{Buckleypneq2}\end{equation} (and in fact clearly $\MC{B}(W, U, A, p)$ is the infimum of all such $C$.)

\begin{proof}

 (b) $\Rightarrow$ (a): By dyadic Littlewood-Paley theory, we need to show that  \begin{align} \int_{\Rd} & \left( \sum_{\varepsilon \in \S} \sum_{I \in \D} \frac{ |V_I (W) A_I ^\varepsilon m_I(U^{-\frac{1}{p}} \vec{f} )| ^2 }{|I|} \chi_I(t) \right)^\frac{p}{2} \, dt
\nonumber \\ & \leq \int_{\Rd}  \left( \sum_{\varepsilon \in \S} \sum_{I \in \D} \frac{ (\|V_I (W) A_I ^\varepsilon (V_I(U))  ^{-1} \| m_I |V_I (U) U^{-\frac{1}{p}} \vec{f}|)^2 }{|I|} \chi_I(t) \right)^\frac{p}{2} \, dt \label{CarEmbedParEst}\\ & \lesssim\|\vec{f}\|_{L^p} ^p \nonumber  \end{align} for any $\vec{f} \in L^p(\Rd;\C)$.

Now let \begin{equation*} \tilde{A} = \sum_{\varepsilon \in \S} \sum_{I \in \D} \|V_I (W) A_I ^\varepsilon (V_I (U)) ^{-1} \| h_I ^\varepsilon \end{equation*} and let \begin{equation*} M_U ' \vec{f} (x) = \sup_{\D \ni I \ni x}  m_I |V_I (U) U^{-\frac{1}{p}} \vec{f}| \end{equation*} Clearly for any $\D \ni I \ni x$ we have that \begin{equation*} m_I |V_I (U) U^{-\frac{1}{p}} \vec{f}| \leq m_I M_U ' \vec{f} \end{equation*} so that \begin{align*}  \eqref{CarEmbedParEst} & \leq \int_{\Rd}  \left( \sum_{\varepsilon \in \S} \sum_{I \in \D} \frac{ (\|V_I (W) A_I ^\varepsilon (V_I (U)) ^{-1} \| m_I( M_U ' \vec{f}))^2 }{|I|} \chi_I(t) \right)^\frac{p}{2} \, dt
\\ & \lesssim \|\pi_{\tilde{A}} M_U ' \vec{f}\|_{L^p} ^p   \lesssim \|A\|_* ^p \|M_U ' \vec{f}\|_{L^p} ^p \end{align*}  where $ \|A\|_*$ is the canonical supremum from condition (b).

However, it is easy to see that \begin{equation*} \|M_U ' \|_{L^p} \lesssim \|U\|_{\text{A}_p} ^{\frac{1}{p - 1}} \end{equation*} by using some simple ideas from \cite{G} (see \cite{IKP}), which means that \begin{equation} \|\Pi_A ^{W, U, p} \|_{L^p \rightarrow L^p} \approx \|A\|_{\ast} \|U\|_{\text{A}_p} ^\frac{1}{p - 1} \label{SharpEmbedThm} \end{equation} and thus completes the proof of $a) \Rightarrow b)$.

(a) $\Rightarrow$ (b): Fixing $J \in \D$, plugging in the test functions $\vec{f} := \chi_J \vec{e}_i$ into $\Pi_A$ for any orthonormal basis $\{\vec{e}_i\}_{i = 1}^n $ of $\C$, and using $(a)$ combined with dyadic Littlewood-Paley theory and elementary linear algebra gives us that \begin{align*} \|\Pi_A ^{W, U, p}\|_{L^p \rightarrow L^p} ^p  |J|  & \gtrsim    \int_{\Rd} \left(  \sum_{\varepsilon \in \S} \sum_{I \in \D} \frac{\|V_I (W) A_I ^\varepsilon m_I (\chi_J U^{- \frac{1}{p}} )\| ^2}{|I|} \chi_I (x) \right)^\frac{p}{2} \, dx \nonumber \\ & \geq  \int_J \left( \sum_{\varepsilon \in \S} \sum_{I \in \D(J)} \frac{\|V_I (W)A_I ^\varepsilon m_I ( U^{- \frac{1}{p}} )\| ^2}{|I|} \chi_I (x) \right)^\frac{p}{2} \, dx \end{align*} which in conjunction with Lemma \ref{RedOp-AveLem} says that \begin{align*}  \sup_{J \in \D} \frac{1}{|J|} \ & \int_J  \left( \sum_{\varepsilon \in \S} \sum_{I \in \D(J)} \frac{\|V_I(W) A_I ^\varepsilon (V_I (U))^{-1} \| ^2}{|I|} \chi_I (x) \right)^\frac{p}{2} \, dx  \\ & \lesssim \sup_{J \in \D} \frac{1}{|J|} \ \int_J  \left( \sum_{\varepsilon \in \S} \sum_{I \in \D(J)} \frac{\|V_I (W) A_I ^\varepsilon V_I ' (U) \| ^2}{|I|} \chi_I (x) \right)^\frac{p}{2} \, dx  \\ & \lesssim   \sup_{J \in \D} \frac{1}{|J|} \ \int_J \left( \sum_{\varepsilon \in \S}  \sum_{I \in \D(J)} \frac{\|V_I (W)A_I ^\varepsilon m_I ( U^{- \frac{1}{p}} )\| ^2}{|I|} \chi_I (x) \right)^\frac{p}{2} \, dx \\ & \lesssim \| \Pi_A ^{W, U, p}\|_{L^p \rightarrow L^p}^p. \end{align*} Condition $(b)$ now follows immediately from Theorem $3.1$ in \cite{NTV}, so that (a) $\Leftrightarrow$ (b) for all $1 < p < \infty$.    \\

 (b) $\Leftrightarrow$ (d): To avoid confusion in the subsequent arguments, we will write $V_I(W) = V_I(W, p)$ to indicate which $p$ the $V_I(W)$ at hand is referring to.  As mentioned before, it is easy to see that $W$ is a matrix A${}_p$ weight if and only if $W^{1 - p'}$ is a matrix A${}_{p'}$ weight and the same for $U$.  Furthermore, one can easily check that we can choose $V_I(W^{1 - p'}, p') = V_I ' (W, p)$ and $V_I ' (W^{1 - p'}, p') =  V_I(W, p),$ and that the same for $U$ holds. Thus, the two equalities above combined with the matrix A${}_p$ condition gives us that \begin{align*} & \sup_{J \in \D} \frac{1}{|J|} \sum_{\varepsilon \in \S} \sum_{I \in \D(J)} \|V_I(U^{1 - p'}, p')( A_I ^\varepsilon) ^* (V_I  (W^{1 - p'}, p'))^{-1}\|^2 \nonumber \\ & \approx   \sup_{J \in \D} \frac{1}{|J|} \sum_{\varepsilon \in \S} \sum_{I \in \D(J)} \|V_I(W, p) A_I ^\varepsilon (V_I  (U, p))   ^{-1} \|^2     \end{align*}  so applying (a) $\Leftrightarrow$ (b) for the quadruplet $(W, U, A, p)$ replaced by $(U^{1-p'}, W^{1-p'}, A^*,  p')$ gives us that (a) $\Leftrightarrow$ (b) $\Leftrightarrow$ (d) for all $1 < p < \infty$. We now prove that (c) $\Rightarrow$ (b) for all $1 < p < \infty$. \\

(c) $\Rightarrow$ (b):  We will in fact show that either of the conditions $\MC{B}(W, U, A, p) < \infty$ or  $\MC{B}(U^{1-p'}, W^{1-p'}, A^*, p') < \infty$ (for any $1 < p < \infty$) implies that (b).  First assume that $\MC{B}(W, U, A, p) < \infty$.   Then by our stopping time, we have that \begin{align*}   & \sum_{\varepsilon \in \S} \sum_{I \in \D(J)}   \|V_I(W) A_I ^\varepsilon (V_I  (U))   ^{-1}  \|^2
 \\ & = \sum_{j = 1}^\infty  \sum_{\varepsilon \in \S} \sum_{K \in \J^{j - 1}(J) } \sum_{I \in \F(K)}   \|V_I(W) A_I ^\varepsilon (V_I  (U))   ^{-1}  \|^2 \\ & \leq \sum_{j = 1}^\infty  \sum_{\varepsilon \in \S} \sum_{K \in \J^{j - 1}(J) } \sum_{I \in \F(K)}   \|V_I(W) A_I ^\varepsilon (V_K(U))^{-1}\| ^2 \|V_K(U)  (V_I  (U))   ^{-1}  \|^2  \\ &     \lesssim  \sum_{j = 1}^\infty  \sum_{\varepsilon \in \S} \sum_{K \in \J^{j - 1}(J) } \sum_{I \in \D(K)}   \|V_I(W) A_I ^\varepsilon (V_K(U))^{-1}\| ^2  \\ & \lesssim \MC{B}(W, U, B, p) \sum_{j = 1}^\infty \sum_{K \in \J^{j - 1}(J) } |K| \\ & \leq |J| \MC{B}(W, U, B, p).\end{align*}   Now to prove that (b) is true when $\MC{B}(U^{1-p'}, W^{1-p'}, A^*, p') < \infty$, notice that replacing the quadruplet $(W, U, A, p)$  by $(U^{1-p'}, W^{1-p'}, A^*,  p')$ gives us that \begin{align*} & \sup_{J \in \D} \frac{1}{|J|} \sum_{\varepsilon \in \S} \sum_{I \in \D(J)} \|V_I(W, p) A_I ^\varepsilon (V_I  (U, p))   ^{-1} \|^2 \\ & \approx \sup_{J \in \D} \frac{1}{|J|} \sum_{\varepsilon \in \S} \sum_{I \in \D(J)} \|V_I(U^{1 - p'}, p')( A_I ^\varepsilon) ^* (V_I  (W^{1 - p'}, p'))^{-1}\|^2 \nonumber \\ & <   \MC{B}(U^{1-p'}, W^{1-p'}, A^*, p').       \end{align*} \noindent   We now prove that  (a) $\Rightarrow$ (c) when $2 \leq p < \infty$ and (d) $\Leftrightarrow$ (c) when $1 < p \leq  2$


 (a) $\Rightarrow$ (c) when $2 \leq p < \infty$: Fix $J \in \D$ and $\vec{e} \in \C$.  If $\vec{f} = U^\frac{1}{p} \chi_J \vec{e}$, then condition (a), the definition of $V_J(U)$, and H\"{o}lder's inequality give us that \begin{align*} |J| |V_J (U) \vec{e}|^p \|\Pi_A \|_{L^p \rightarrow L^p} ^p & \gtrsim \int_{\Rd} \left( \sum_{I \in \D} \sum_{\varepsilon \in \S} \frac{|V_I (W) A_I ^\varepsilon m_I (\chi_J \vec{e}) |^2}{|I| } \chi_I (t) \right)^\frac{p}{2} \, dt \\ & \geq |J| \left[\frac{1}{|J|} \int_J \left( \sum_{\varepsilon \in \S} \sum_{I \in \D(J)} \frac{|V_I (W) A_I ^\varepsilon  \vec{e}|^2 }{|I| } \chi_I (t) \right)^\frac{p}{2} \, dt \right] \\ & \geq |J| \left[ \frac{1}{|J|} \sum_{\varepsilon \in \S} \sum_{I \in \D(J)} |V_I (W) A_I ^\varepsilon \vec{e}|^2 \right]^\frac{p}{2} \end{align*} which proves (c) after replacing $\V{e}$ with $(V_J(U))^{-1} \V{e}$, and in fact shows that (a) $\Leftrightarrow$ (b)   $\Leftrightarrow$ (c) $\Leftrightarrow$ (d) when $ 2 \leq p < \infty$.  We now complete the proof when $1 < p \leq 2$.  \\

 (d) $\Leftrightarrow$ (c) when $1 < p \leq  2$:  Since $2 \leq p' < \infty$, we have that (d) $\Leftrightarrow$ (c) when $1 \leq p < 2$ by replacing the quadruplet $(W, U, A, p)$  with $(U^{1-p'}, W^{1-p'}, A ^*, p')$ and utilizing (a) $\Rightarrow$ (c). \end{proof}

While it is clear from the proof above, we shall  point out that the sole reason for the two different conditions in (c) is that we are only able to prove that (a) $\Rightarrow \MC{B}(W, U, A, p) < \infty$ when $2 \leq p < \infty$ and (d) $\Rightarrow \MC{B}(U^{1-p'}, W^{1-p'}, A^*, p') < \infty$ when $1 <  p \leq 2$ (since $2 \leq p' < \infty$) for the simple reason that we crucially require the use of H\"{o}lder's inequality with respect to the exponent $\frac{p}{2}$ and $\frac{p'}{2}$ respectively.  Furthermore, for this reason, we have when $p = 2$ that $\MC{B}(W, U, A, 2) < \infty$ is equivalent to $\MC{B}(U^{-1}, W^{-1}, A^*, 2) < \infty$.

Moreover, it is instructive and quite interesting to compare Theorem \ref{CarEmbedThm} when $p = 2$ to Theorem $3.1$ of \cite{HLW0} in the scalar setting.  In particular it was shown in \cite{HLW0} that a scalar symbolled paraproduct $\pi_b : L^2(u) \rightarrow L^2(w)$ for two scalar A${}_2$ weights $w$ and $u$ if and only if \begin{equation} \label{HLWCond} \sup_{J \in \D}  \frac{1}{u^{-1} (J)}\sum_{\varepsilon \in \S} \sum_{I \in \D(J)} |b_I ^\varepsilon|^2 (m_I (u^{-1}))^2 m_I w < \infty.  \end{equation}

  Although we will not need it to prove the main results of this paper, we will now prove that a matrix weighted version of \eqref{HLWCond} is equivalent to the boundedness of $\Pi_B ^{W, U, 2}$ on $L^2$ (and clearly a more general statement can be said regarding similar matrix sequences that are not necessarily Haar coefficients), which of course generalizes Theorem $3.1$ in \cite{HLW0} to the matrix $p = 2$ setting.

   \begin{proposition} \label{HLWIsralProp}  $\Pi_B ^{W, U, 2}$ is bounded on $L^2$ if and only if there exists $C > 0$ independent of $J$ where \begin{equation} \label{HLWICond} \sum_{\varepsilon \in \S} \sum_{I \in \D(J)} m_I(U^{-1} ) (B_I^\varepsilon)^* (m_I W) B_I ^\varepsilon m_I (U ^{-1})  \leq C (U^{-1} (J) ). \end{equation}  \end{proposition}

   \begin{proof}Assume first that $\Pi_B ^{W, U, 2}$ is bounded on $L^2$.   Using the testing function $\vec{f} = U^{-\frac12} \chi_J \vec{e}$ for any vector $\vec{e}$ gives us \begin{equation*}  \sum_{\varepsilon \in \S} \sum_{I \in \D(J)} |(m_I W)^{\frac12} B_I ^\varepsilon m_I (U ^{-1})   \vec{e} |^2 \leq C \ip{ U^{-1} (J) \vec{e}}{ \vec{e}}_{\C} \end{equation*} where $C = \|\Pi_B ^{W, U, 2}\|_{L^2 \rightarrow L^2} $.

Conversely,  by Theorem $1.2$ in \cite{CT}, \eqref{HLWICond} immediately implies that \begin{equation*} \sum_{\varepsilon \in \S} \sum_{I \in \D} \| [(B_I^\varepsilon)^* (m_I W) B_I ^\varepsilon]^\frac12 m_I (U^{-\frac12 } \vec{f })\| ^2  \leq C \|\vec{f}\|_{L^2} ^2 \end{equation*} \noindent  for some $C$ independent of $U$ and $\vec{f}$.  Plugging in test functions of the form $\vec{f} =  \chi_J \vec{e}$ for any vector $\vec{e}$  in conjunction with Lemma \ref{RedOp-AveLem}   gives \begin{align*} \sum_{\varepsilon \in \S} & \sum_{I \in \D(J)}  \|(m_I W)^{\frac12}  B_I ^\varepsilon (m_I U)^{-\frac12}  \|^2 \\ & \approx \sum_{\varepsilon \in \S} \sum_{I \in \D(J)} \|(m_I W)^{\frac12}  B_I ^\varepsilon m_I (U^{-\frac12}) \|^2 \\ & = \sum_{\varepsilon \in \S} \sum_{I \in \D(J)} \|m_I (U^{-\frac12}) (B_I ^\varepsilon)^* m_I W   B_I ^\varepsilon    m_I (U^{-\frac12}) \| \\ & = \sum_{\varepsilon \in \S} \sum_{I \in \D} \| [(B_I^\varepsilon)^* (m_I W) B_I ^\varepsilon]^\frac12 m_I (U^{-\frac12 } )\| ^2  \leq C |J|. \qedhere  \end{align*}\end{proof}

As in \cite{IKP, HLW}, we can provide a much cleaner continuous BMO condition that characterizes the boundedness of paraproducts.

\begin{corollary} \label{BMOCont} If $1 < p < \infty, W$ and $U$ are matrix A${}_p$ weights, and $\D$ is a dyadic grid, then the following are equivalent:

\begin{itemize}{}{}

\item [(a')] $\pi_B : L^p(U) \rightarrow L^p(W)$ boundedly \\
\item[(b')]$\displaystyle \sup_{J \in \D} \frac{1}{|J|} \int_J \|W^\frac{1}{p} (x) (B(x) - m_J B) (V_J(U)) ^{-1}  \|^p \, dx < \infty $
\item[(c')] $\displaystyle \sup_{J \in \D}  \frac{1}{|J|} \int_J \|U^{-\frac{1}{p}} (x) (B^* (x) - m_J B^*) (V_J ' (W))^{-1}  \|^{p'} \, dx < \infty.   $ \\

\end{itemize}
\end{corollary}

\begin{proof}
Assume that (a') is true.  As was mentioned before, $\pi_B : L^p(U) \rightarrow L^p(W)$  boundedly if and only if $\Pi_B ^{W, U, p}$ is bounded on $L^p$, so \eqref{LpEmbedding} gives us that

\begin{align*}    &  \int_J \|W^\frac{1}{p} (x) (B(x) - m_J B) (V_J(U)) ^{-1}  \|^p \, dx \\ & \approx   \sum_{i = 1}^n \sup_{J \in \D}  \ \int_{\Rd} |W^\frac{1}{p} (x) \chi_J (x) (B(x) - m_J B) (V_J (U)) ^{-1} \V{e}_i |^p \, dx \\ & \approx \sum_{i = 1}^n     \ \int_{\Rd} \left(\sum_{\varepsilon \in \S} \sum_{I \in \D(J)} \frac{|V_I(W) B_I ^\varepsilon (V_J (U))^{-1} \V{e}_i |^2}{|I|} \chi_I (x) \right)^\frac{p}{2} \, dx \\ & \leq \sum_{i = 1}^n   \int_{\Rd} \left(\sum_{I \in \D(J)} \sum_{\epsilon \in \S} \frac{|V_I (W) B_I ^\varepsilon m_I (U^{-\frac{1}{p}} \{ \chi_J  U^\frac{1}{p} (V_J (U)) ^{-1} \V{e}_i\}) |^2}{|I|} \chi_I (x) \right)^\frac{p}{2} \, dx  \\ & \lesssim \sum_{i = 1}^n   |J| \|\Pi_B ^{W, U, p} \chi_J (U^{\frac{1}{p}}  (V_J(U)) ^{-1} \V{e}_i)\|_{L^p} ^p  \lesssim \|\Pi_B ^{W, U, p}\|_{L^p} ^p
\end{align*}  which means that (b') is true.   Since $\Pi_B ^{W, U, p}$ is bounded on $L^p$ if and only if $\Pi_{B^*} ^{U^{1-p'}, W^{1-p'}, p'}$ is bounded on $L^{p'}$, (c') also immediately follows if $\pi_B : L^p(U) \rightarrow L^p(W)$ is bounded.

We now prove that (b') implies that $\Pi_B ^{W, U, p}$ is bounded on $L^p$ which will clearly also gives us that (c') implies that $\Pi_{B^*} ^{U^{1-p'}, W^{1-p'}, p'}$ is bounded on $L^{p'}$, so that (b') $\Rightarrow$ (a') and (c') $\Rightarrow$ (a').    Now if (b') is true then \eqref{LpEmbedding} gives us that for any $\vec{e} \in \C$ \begin{align*}  \sup_{J \in \D} \ \frac{1}{|J|} &  \int_J |W^\frac{1}{p} (x) (B(x) - m_J B) (V_J (U))^{-1} \vec{e}   |^p \, dx \\ & \approx    \sup_{J \in \D}  \ \frac{1}{|J|} \int_J \left(\sum_{\varepsilon \in \S} \sum_{I \in \D(J)} \frac{|V_I (W) B_I ^\varepsilon (V_J (U))^{-1} \vec{e}|^2}{|I|} \chi_I (x)   \right) ^\frac{p}{2} \, dx    \end{align*} and in particular if $2 \leq p < \infty$ then by H\"{o}lder's inequality we have
\begin{align*} \sup_{J \in \D} \ \frac{1}{|J|} &  \int_J |W^\frac{1}{p} (x) (B(x) - m_J B) (V_J (U))^{-1} \vec{e}   |^p \, dx \\& \gtrsim  \sup_{J \in \D}  \ \left( \frac{1}{|J|} \sum_{\varepsilon \in \S} \sum_{I \in \D(J)} |V_I (W) B_I ^\varepsilon (V_J (U))^{-1} \vec{e}|^2 \right)^\frac{p}{2}. \end{align*}
 However, if $1 < p \leq 2$ then \begin{align}
&  \frac{1}{|J|} \int_{J} \left( \sum_{\varepsilon \in \S} \sum_{I \in \D(J)} \frac{\|V_I (W) B_I ^\varepsilon (V_I(U)) ^{-1} \| ^2}{|I|} \chi_I (x)  \right) ^\frac{p}{2} \, dx \nonumber \\ & =
\frac{1}{|J|} \int_J \left( \sum_{j = 1} ^\infty \sum_{K \in \J^{j - 1} (J)} \sum_{\varepsilon \in \S} \sum_{I \in \F(K)}  \frac{\|V_I (W) B_I ^\varepsilon (V_I (U)) ^{-1} \|^2 }{|I|} \chi_I (x)  \right) ^\frac{p}{2} \, dx \nonumber \\ & \lesssim \frac{1}{|J|} \int_J \left( \sum_{j = 1} ^\infty \sum_{K \in \J^{j - 1} (J)} \sum_{\varepsilon \in \S} \sum_{I \in \F(K)}  \frac{\|V_I (W) B_I ^\varepsilon (V_K (U)) ^{-1} \|^2 }{|I|} \chi_I (x) \right) ^\frac{p}{2} \, dx \nonumber \\ & \leq \frac{1}{|J|}  \int_J \left(  \sum_{j = 1} ^\infty \sum_{K \in \J^{j - 1} (J)}  \sum_{\varepsilon \in \S} \sum_{I \in \D(K)} \frac{\|V_I(W) B_I ^\varepsilon  (V_K (U)) ^{-1} \|^2 }{|I|} \chi_I (x) \right) ^ \frac{p}{2} \, dx \nonumber \\ & \leq  \frac{1}{|J|}     \sum_{j = 1} ^\infty \sum_{K \in \J^{j - 1} (J)}  \int_K \left( \sum_{\varepsilon \in \S} \sum_{I \in \D(K)} \frac{\|V_I(W) B_I ^\varepsilon  (V_K (U)) ^{-1} \|^2 }{|I|} \chi_I (x)   \right) ^ \frac{p}{2} \, dx \nonumber \\ & \lesssim \frac{1}{|I|}   \sum_{j = 1} ^\infty \sum_{K \in \J^{j - 1}  (I)} |K|   \nonumber \\& \lesssim \sum_{j = 1}^\infty 2^{-(j - 1)} < \infty \nonumber \end{align} which by Theorem 3.1 in \cite{NTV} says that (b) in Theorem \ref{CarEmbedThm} is true, which implies that (a') is true.  \end{proof}

From now on we will say that $B \in \BMOWUAltD$ for a dyadic grid $\D$ if $B$ satisfies either of the conditions in Corollary \ref{BMOCont} with respect to $\D$,  any of the equivalent conditions in Theorem \ref{CarEmbedThm} with respect to $\D$, or \eqref{HLWICond} with respect to $\D$.  In the last section we will show that $\BMOWU$  coincides with the union of $\BMOWUAltD$ over a finite number of dyadic grids $\D$.

\section{Two weight characterization of Riesz transforms}

We will now prove Theorem \ref{CommutatorThm} but in terms of $\BMOWUAlt$, which we will define as the the union of $\BMOWUAltD$ over all dyadic grids $\D$ (which as usual will be shown to coincide with the union of $\BMOWUAltD$ over a finite number of dyadic grids $\D$). Before we do this we will need the following simple but nonetheless interesting characterization of matrix Haar multipliers. Note that the one matrix weighted characterization of these Haar multipliers was first proved in \cite{IKP} and that a sharper result (in terms of the A${}_2$ dependency) was soon after proved in \cite{BPW} when $p = 2$.    \begin{proposition} \label{HaarMultThm} Let $1 < p < \infty$ and let $W$ be a matrix A${}_p$ weight.  If $\D$ is any dyadic grid and $A := \{A_I ^\varepsilon  : I \in \D, \varepsilon \in \S\}$ is a sequence of matrices, then the Haar multiplier \begin{equation*} T_A \vec{f} := \sum_{I \in \D} \sum_{\varepsilon \in \S} A_I ^\varepsilon {\vec{f}}_I ^\varepsilon  h_I ^\varepsilon \end{equation*} is bounded from $L^p(U)$ to $L^p(W)$ if and only if \begin{equation*} \sup_{I \in \D, \varepsilon \in \S} \|V_I(W) A_I ^\varepsilon (V_I (U))^{-1}\| < \infty. \end{equation*}  \end{proposition}

 \begin{proof} If $M = \sup_{I \in \D, \varepsilon \in \S} \|V_I(W) A_I ^\varepsilon (V_I (U))^{-1}\| < \infty$, then two applications of $(\ref{LpEmbedding})$ give us that \begin{align*} \|T_A \vec{f} \|_{L^p(W)} ^p & \approx \inrd \left(\sum_{I \in \D} \sum_{\varepsilon \in \S} \frac{|V_I (W) A_I ^\varepsilon {\vec{f}}_I ^\varepsilon |^2}{|I|} \chi_I (x) \right)^\frac{p}{2} \, dx \\ & \leq \inrd \left(\sum_{I \in \D} \sum_{\varepsilon \in \S} \frac{\|V_I(W) A_I ^\varepsilon (V_I(U)) ^{-1} \|^2 | V_I (U) {\vec{f}}_I ^\varepsilon|^2}{|I|} \chi_I (x) \right)^\frac{p}{2} \, dx \\ & \leq M^p \inrd \left(\sum_{I \in \D} \sum_{\varepsilon \in \S} \frac{ | V_I (U) {\vec{f}}_I ^\varepsilon |^2}{|I|} \chi_I (x) \right)^\frac{p}{2} \, dx \\ &  \approx M^p \|\vec{f}\|_{L^p(U)} ^p. \end{align*}

For the other direction, fix some $J_0 \in \D$  and $\varepsilon' \in \S$, and let $J_0 ' \in \D(J_0)$ with $\ell(J_0 ') = \frac{1}{2} \ell(J_0)$. Again by  \eqref{LpEmbedding} we have that \begin{equation} \label{NecHaarMultEq}  \int_{\Rd} \left(\sum_{I \in \D} \sum_{\varepsilon \in \S} \frac{|V_I(W) A_I ^\varepsilon (U^{-\frac{1}{p}} \vec{f} )_I ^\varepsilon |^2 }{|I|} \chi_I (x) \right) ^\frac{p}{2} \, dx \lesssim \|\vec{f}\|_{L^p}^p.   \end{equation}  Plugging $\vec{f} := \chi_{J_0 '} \vec{e}$ for any $\vec{e} \in \C$ into (\ref{NecHaarMultEq}) and noticing that \begin{equation*}  (U^{-\frac{1}{p}} \vec{f})_{J_0} ^{\varepsilon '} = (U^{-\frac{1}{p}} \chi_{J_0 '} \vec{e} )_{J_0} ^{\varepsilon'} = \pm 2^{-\frac{d}{2}} |J_0| ^\frac{1}{2} m_{J_0 '} (U^{- \frac{1}{p}}) \end{equation*} easily gives us (in conjunction with Lemma \ref{RedOp-AveLem}) that
\begin{align*} |V_{J_0} (W) & A_{J_0} ^{\varepsilon'} V_{J_0 '} '(U) \vec{e}  |^p   \\ & \lesssim  |V_{J_0} (W) A_{J_0} ^{\varepsilon'} m_{J_0 '} (U^{- \frac{1}{p}}) \vec{e} |^p  \\  & \approx
|J_0|^{-\frac{p}{2}} |V_{J_0} (W) A_{J_0} ^{\varepsilon '} (U^{-\frac{1}{p}} \vec{f})_{J_0} ^{\varepsilon '} \vec{e} |^p \\ & = \frac{1}{|J_0|} \inrd \left(\frac{|V_{J_0} (W) A_{J_0} ^{\varepsilon '} (U^{-\frac{1}{p}} \vec{f})_{J_0} ^{\varepsilon '} \vec{e} |^2}{|J_0|} \chi_{J_0} (x) \right)^\frac{p}{2} \, dx \\ & \leq \frac{1}{|J_0|} \inrd \left(\sum_{\varepsilon \in \S} \sum_{I \in \D} \frac{|V_{I} (W) A_{I} ^{\varepsilon } (U^{-\frac{1}{p}} \vec{f} )_{I} ^{\varepsilon } |^2}{|I|} \chi_{I} (x) \right)^\frac{p}{2} \, dx \\ & \lesssim \|T_A\|_{L^p(U) \rightarrow L^p(W)} ^p. \end{align*}   Using the definition of $V_{J_0 '} ' $ and summing over all of the $2^d$ first generation children $J_0 '$ of $J_0$ finally (after taking the supremum over $J_0 \in \D$) gives us that

\begin{equation*} \sup_{J,  \ \varepsilon } \|V_J(W) A_J ^\varepsilon (V_J (U))^{-1} \|  \lesssim  \sup_{J , \ \varepsilon } \|V_{J} (W)  A_{J} ^\varepsilon  V_{J}'  (U) \|     \lesssim  \|T_A \|_{L^p (W) \rightarrow L^p (U)}  \qedhere \end{equation*}   \end{proof}

We now prove sufficiency in Theorem \ref{CommutatorThm} with respect to $\BMOWUAlt$. As in \cite{LPPW, IKP,  HLW}. the starting point is the fact that any of the Riesz  transforms are in the $L^2-$ SOT convex hull of the so called ``Haar shifts" which are defined by $ Q_\sigma h_I ^\varepsilon = h_{\sigma(I)} ^{\sigma(\varepsilon)} $  and (slightly abusing notation in the obvious way) $\sigma : \D \times \S \rightarrow \D \times \S$ where $2\ell(\sigma(I)) =  \ell(I)$ and $\sigma(I) \subseteq I$ for each $I \in \D$.  Fixing $\sigma$ and letting $Q = Q_{\sigma}$,  it is then enough to get an $L^p(W)$ bound on each $[B, Q]$.

\begin{thm} If $W$ and $U$ are matrix A${}_p$ weights, $B$ is locally integrable,  and $R$ is any of the Riesz transforms, then $[R, B]$ is bounded from $L^p(U)$ to $L^p(W)$ if  $B \in \BMOWUAlt$.  \end{thm}

\begin{proof} As in \cite{IKP} we use the decomposition in \cite{LPPW}.   First, write \begin{equation*} B = \sum_{I' \in \D} \sum_{\varepsilon' \in \S} B_{I'} ^{\varepsilon'} h_{I'} ^{\varepsilon '}, \ \ \ \vec{f} = \sum_{I \in \D} \sum_{\varepsilon \in \S} \vec{f}_{I} ^{\varepsilon} h_{I} ^{\varepsilon }\end{equation*} so that \begin{align*} [B, Q] \vec{f} & = \sum_{I \in \D} \sum_{\varepsilon \in \S}  \left( B \vec{f}_I ^\varepsilon Q h_I ^\varepsilon - Q (B h_I ^\varepsilon) \vec{f}_I ^\varepsilon \right) \\ & = \sum_{I, I' \in \D} \sum_{\varepsilon, \varepsilon' \in \S}  \left(  B_{I'} ^{\varepsilon'}   h_{I'} ^{\varepsilon '} (Q h_I ^\varepsilon) \vec{f}_I ^\varepsilon - B_{I'} ^{\varepsilon '} (Q  h_{I'} ^{\varepsilon '}  h_I ^\varepsilon) \vec{f}_I ^\varepsilon \right)  \\ & = \sum_{I, I' \in \D} \sum_{\varepsilon, \varepsilon' \in \S}  B_{I'} ^{\varepsilon'} \left([h_{I'} ^{\varepsilon '}, Q]h_I ^\varepsilon \right) \vec{f}_{I} ^\varepsilon \end{align*}

\noindent Clearly there is no contribution if $I \cap I' = \emptyset$ and otherwise we have that

\begin{equation} \label{e.cases}
[  {h _{I'} ^{\varepsilon '}},  Q]h _{I} ^{\varepsilon}
=
\begin{cases}
0  &   I\subsetneq I'
\\
\pm\abs{I} ^{-1/2}  h ^{\sigma (\varepsilon)} _{\sigma (I)}
- Q (h ^{\epsilon'} _{I} h ^{\epsilon} _{I} )
& I=I'
\\
h ^{\varepsilon'} _{\sigma (I)}h ^{\sigma(\varepsilon)} _{\sigma(I)}
\pm \abs{I} ^{-1/2} h ^{\sigma (\varepsilon ')} _{\sigma ^2(I)}
&  I'=\sigma (I)
\\
  h ^{ \varepsilon '} _{I'} Q (h_I ^\varepsilon  ) - Q (h_I ^\varepsilon  h ^{ \varepsilon '} _{I'})
& I'\subsetneq I  \textup{ and } I' \neq \sigma(I)

\end{cases}
\end{equation}

Note that we can disregard sign changes thanks to the unconditionality of Theorem \ref{CarEmbedThm},  \eqref{LpEmbedding}, and Proposition \ref{HaarMultThm}, and we will not comment on this further in the proof.  When $I=I'$ we need to bound the two sums
\begin{equation}\sum_I\sum_{\varepsilon, \varepsilon'\neq\vec{1}}  B_{I} ^{\varepsilon'}\vec{f}_{I} ^\varepsilon |I|^{-1/2} h_{\sigma(I)}^{\sigma(\varepsilon)}  \;\textup{ and }\; Q\left(\sum_I\sum_{\varepsilon, \varepsilon'\neq\vec{1}} B_{I} ^{\varepsilon'}  \vec{f}_{I} ^\varepsilon |I|^{-1/2}h_{I}^{\psi_{\varepsilon'}(\varepsilon)}\right). \label{DiagCommTerm}\end{equation} where $\psi_{\varepsilon'}(\varepsilon)$ is the signature defined by

\begin{equation*} h_I ^{\psi_{\varepsilon'} (\varepsilon)} =  |I|^\frac{1}{2} h_I^{\varepsilon} h_I^{\varepsilon' } \end{equation*} which is obviously cancellative if and only if $\varepsilon \neq \varepsilon'$.

 However,  if $B \in \BMOWU$ then condition $(b)$ in Theorem \ref{CarEmbedThm} obviously tells us that for $\epsilon, \epsilon'$ fixed and $\tilde{J}$ being the parent of $J \in \D$

 \begin{equation*} \sup_{J \in \sigma(\D)}  \| V_{J} (W) (|\tilde{J}| ^{-\frac12} B_{\tilde{J}} ^{\epsilon'}) (V_{\tilde{J}} (U))^{-1} \| \lesssim  \sup_{J \in \sigma(\D)}  \| V_{\tilde{J}} (W) (|\tilde{J}| ^{-\frac12} B_{\tilde{J}} ^{\epsilon'}) (V_{\tilde{J}} (U)) ^{-1} \|  < \infty \end{equation*} so that the first sum in \eqref{DiagCommTerm} can be estimated in a manner that is very similar to the proof of sufficiency in Theorem \ref{HaarMultThm} (that is, using \eqref{LpEmbedding} twice).

Note that the second sum of \eqref{DiagCommTerm} when $\varepsilon \neq \varepsilon'$ is also ``Haar multiplier like" and can be estimated in exactly the same way as the first sum in \eqref{DiagCommTerm} since $Q : L^p(W) \rightarrow L^p(W)$ boundedly (see \cite{IKP}).   On the other hand, when $\epsilon = \epsilon'$ the second sum of \eqref{DiagCommTerm} becomes   \begin{equation*} Q \left(\sum_{I \in \D} \sum_{\varepsilon \in \S}  B_{I} ^{\varepsilon}  \vec{f}_{I} ^{\varepsilon} \frac{\chi_I}{|I|}  \right) = Q (\pi_{B^*} )^* \vec{f}\end{equation*}

\noindent   However, since \begin{align*} \|V_I(W, p) B_I ^\varepsilon (V_I(U, p))^{-1} \| & = \|(V_I(U, p))^{-1} (B_I ^\varepsilon)^* V_I(W, p) \| \\ & \approx   \|V_I '(U, p) (B_I ^\varepsilon)^* (V_I '(W, p))^{-1}  \| \\ & \approx   \|V_I (U^{1-p'}, p') (B_I ^\varepsilon)^* (V_I (W^{1-p'}, p'))^{-1}  \| \end{align*}

\noindent we have that $B \in \BMOWU$ if and only if $\pi_{B^*} : L^{p'} (W^{1-p'}) \rightarrow L^{p'} (U^{1 -p'}) $ is bounded, which is obviously true if and only if $(\pi_{B^*})^* : L^{p} (U) \rightarrow L^{p} (W) $ is bounded so that  $Q (\pi_{B^*} )^* :  L^{p} (U) \rightarrow L^{p} (W)$ is bounded since $Q : L^p(W) \rightarrow L^p(W)$ is bounded (see \cite{IKP}).

We now look at the case when $I'=\sigma (I)$ which clearly gives us two sums corresponding to the two terms in \eqref{e.cases}.  For the first term, we obtain the sum
\begin{align*} \sum_{I  \in \D} \sum_{\varepsilon, \varepsilon' \in \S}  B_{\sigma(I)} ^{\varepsilon'} h ^{\varepsilon'} _{\sigma (I)}h ^{{\sigma(\varepsilon)}} _{\sigma(I)} \vec{f}_{I} ^\varepsilon & =  \sum_{I  \in \D} \sum_{\varepsilon \in \S}   B_{\sigma(I)} ^{\sigma(\varepsilon)}  \vec{f}_{I} ^\varepsilon \frac{\chi_{\sigma(I)}}{|\sigma(I)|}  \\ & + \sum_{I  \in \D}  \sum_{\varepsilon \in \S} \sum_{\varepsilon' \neq \sigma(\varepsilon) }  |I|^{-\frac12} B_{\sigma(I)} ^{\varepsilon'} h_{\sigma(I)} ^ {\psi_{\varepsilon'}(\sigma(\varepsilon))} \vec{f}_{I} ^\varepsilon.\end{align*}

\noindent However, a simple computation gives us  \begin{equation*} \sum_{I  \in \D} \sum_{\varepsilon \in \S}   B_{\sigma(I)} ^{\sigma(\varepsilon)}  \vec{f}_{I} ^\varepsilon \frac{\chi_{\sigma(I)}}{|\sigma(I)|}    = (\pi_{B^*} )^* Q \vec{f} \end{equation*} which is bounded from $ L^{p} (U)$ to $ L^{p} (W)$.  Also, the second sum is again  ``Haar multiplier like" and can be estimated in easily in a manner that is similar to the proof of sufficiency for Theorem \ref{HaarMultThm}.

 Furthermore, for the second sum in the two terms when $I' = \sigma(I)$, we need to bound
\begin{equation*} \sum_{I \in \D} \sum_{\varepsilon, \varepsilon' \in \S}  B_{\sigma(I)} ^{\varepsilon'} \abs{I} ^{-1/2} h ^{\sigma (\varepsilon ')} _{\sigma ^2(I)}  \vec{f}_{I} ^\varepsilon \end{equation*} which yet again is  ``Haar multiplier like" and can be estimated in a manner that is similar to the proof of sufficiency for Theorem \ref{HaarMultThm}

To finally finish the proof of sufficiency we bound the triangular terms.  First, if $I \supsetneq I'$ then obviously $h_I ^\varepsilon$ is constant on $I'$.  Thus,  \begin{align*}  \sum_{I'\in \D} \sum_{I \supsetneq I'} \sum_{\varepsilon, \varepsilon' \in \S}  B_{I'} ^{\varepsilon'} Q (h_I ^\varepsilon  h ^{ \varepsilon '} _{I})  \vec{f}_I ^\varepsilon & =  \sum_{I' \in \D} \sum_{ \varepsilon'  \in \S} B_{I'} ^{\varepsilon'} Q(h ^{\varepsilon '} _{I'}) \sum_{I \supsetneq I'} \sum_{\varepsilon \in \S}    \vec{f}_{I} ^\varepsilon h_I ^\varepsilon  \\ & = \sum_{I' \in \D}\sum_{ \varepsilon' \in \S} B_{I'} ^{\varepsilon'} Q h_{I'} ^{\varepsilon'} m_{I'} \vec{f}  = Q \pi_B \vec{f} \end{align*}

\noindent Now clearly $h ^{ \varepsilon '} _{I'} Q (h_I ^\varepsilon  ) = 0$ if $I' \cap \sigma(I) = \emptyset$. Furthermore, since  $I \supsetneq I'$  and $I' \neq \sigma(I)$, we must have $\sigma(I) \supsetneq I'$  so that  \begin{align*}  \sum_{I'\in \D} \sum_{I \supsetneq I'} \sum_{\varepsilon, \varepsilon' \in \S}  B_{I'} ^{\varepsilon'} h ^{ \varepsilon '} _{I'} Q (h_I ^\varepsilon  ) \vec{f}_{I} ^\varepsilon  & = \sum_{I'\in \D} \sum_{\varepsilon' \in \S}  B_{I'} ^{\varepsilon'} h ^{ \varepsilon '} _{I'} \sum_{I: \sigma(I) \supsetneq I'}  \sum_{ \varepsilon \in \S} h_{\sigma(I)}  ^{\sigma(\varepsilon)}  \vec{f}_{I} ^\varepsilon \\ & = \sum_{I' \in \D} \sum_{ \varepsilon' \in \S} B_{I'} ^{\varepsilon'}  h_{I'} ^{\varepsilon'} m_{I'} (Q\vec{f}) \\ & = \pi_B Q \vec{f} \end{align*}

\noindent which is obviously bounded from $ L^{p} (U)$ to $ L^{p} (W)$. \end{proof}

Let us make one important remark regarding the above theorem.  A knowledgable reader might wonder why we have not utilized the by now classical Hyt\"{o}nen decomposition theorem (see \cite{H}) to prove sufficiency in Theorem \ref{CommutatorThm} for general CZOs (which was done in \cite{HLW} in the scalar setting).  First, this would require one to prove a two matrix weighted H${}^1$-BMO duality result when $p \neq 2$, which while possible, seems quite tricky to even formulate.  Second, and perhaps more interestingly, it appears to be rather difficult, even when $p = 2$, to prove sub-exponential matrix weighted bounds for Haar shifts (in terms of their complexity).  Thus, even proving sufficiency in Theorem \ref{CommutatorThm} for general CZOs when $p = 2$ looks to be highly nontrivial. Intriguingly,  note that the boundedness of general ``cancellation CZOs" (i.e. CZOs where $T1 = T^*1 = 0)$ on matrix weighted $L^p$ was proved in \cite{NT} by utilizing pre-Hyt\"{o}nen probabilistic surgical ideas that remove singularities in a way that is similar to Hyt\"{o}nen's arguments, but does \textit{not} involve a reorganization into Haar shifts.  Furthermore, note that the author used similar pre-Hyt\"{o}nen probabilistic surgical ideas in \cite{I2} to prove matrix weighted bounds for certain matrix kernelled CZOs.

 We now prove necessity in terms of \BMOWUAlt.  As in \cite{IKP}, we can use the simple ideas in \cite{J} to prove necessity for a wider class of CZOs than just the Riesz transforms.  More precisely,

 \begin{thm} \label{RieszTransNec}  Let $K : \mathbb{R}^d \backslash \{0\} \rightarrow \mathbb{R}^d$  be not identically zero, be homogenous of degree $-d$, have mean zero over the unit sphere $\partial \mathbb{B}_d$, and satisfy $K \in C^{\infty} (\partial \mathbb{B}_d)$ (so in particular $K$ could be any of the Riesz kernels).  If $T$ is the (convolution) CZO associated to $K$, then we have that $[T, B]$ being bounded from $L^p(U)$ to $L^p(W)$ implies that $B \in \BMOWUAlt$.   \end{thm}

\begin{proof}  We will prove (b') in Corollary \ref{BMOCont}.   By assumption, there exists $z_0 \neq 0 $  and $\delta > 0$ where $\frac{1}{K(x)}$ is smooth on $|x - z_0| < \sqrt{d} \delta$, and thus can be expressed as an absolutely convergent Fourier series \begin{equation*} \frac{1}{K(x)} = \sum a_k e^{i v_k \cdot x} \end{equation*} \noindent for $|x - z_0| < \sqrt{d} \delta$ (where the exact nature of the vectors $v_k$ is irrelevant.)  Set $z_1 = \delta^{-1} z_0$.  Thus, if $|x - z_1| < \sqrt{d}$, then we have by homogeneity

\begin{equation*} \frac{1}{K(x)} = \frac{\delta^{-d}}{K(\delta x)} = \delta^{-d} \sum a_n e^{i v_k \cdot (\delta x)}. \end{equation*} \noindent Now for any cube $Q = Q(x_0, r)$ of side length $r$ and center $x_0$, let $y_0 = x_0 - rz_1$ and $Q' = Q(y_0, r)$ so that $x \in Q$ and $y \in Q'$ implies that \begin{equation*} \left|\frac{x - y}{r} - z_1\right|  \leq \left| \frac{x - x_0}{r} \right| + \left| \frac{y - y_0}{r} \right|\leq \sqrt{d}. \end{equation*}

Let \begin{equation*} S_Q (x) = \chi_Q (x) \frac{(W^\frac{1}{p}(x)  (B(x) - m_{Q'} B ) (V_Q (U)) ^{-1}   )^* } {\| W^\frac{1}{p}(x)  (B(x) - m_{Q'} B ) (V_Q (U)) ^{-1}   \|} \end{equation*}  so that for $x \in Q$ \begin{align}\frac{1}{r^d} &  \left\|\inrd  W^\frac{1}{p}(x)  (B(x) -  B(y) ) (V_Q (U)) ^{-1}  \frac{r^d K(x - y)}{K(\frac{x-y}{r})}  S_Q (x)  \chi_{Q'} (y) \, dy  \right\| \label{CommEst1} \\ & =   \frac{1}{r^d} \left\| \int_{Q'}   W^\frac{1}{p}(x)  (B(x) -  B(y) ) (V_Q (U)) ^{-1}  \right.   \nonumber \\ & \qquad \left. \times \frac{(W^\frac{1}{p}(x)  (B(x) - m_{Q'} B ) (V_Q (U)) ^{-1}   )^* } {\| W^\frac{1}{p}(x)  (B(x) - m_{Q'} B ) (V_Q (U)) ^{-1}   \|}  \, dy \right\| \nonumber \\ &  = \| W^\frac{1}{p}(x)  (B(x) - m_{Q'} B ) (V_Q (U)) ^{-1} \|. \nonumber \end{align}

However, \begin{align*} & \eqref{CommEst1} \\ &  \lesssim \sum_k |a_k| \left\|  \left(\inrd W^\frac{1}{p}(x)  (B(x) - B(y) ) K(x - y)  e^{- i \frac{\delta}{r} v_k \cdot y} (V_Q (U)) ^{-1}  \chi_{Q'} (y) \, dy \right) \right. \nonumber \\ & \qquad \times \left. S_Q (x)  e^{ i \frac{\delta}{r} v_k \cdot x} \right\|  \\ & \leq  \sum_k  |a_k| \left\|  \left(\inrd W^\frac{1}{p}(x)  (B(x) -  B(y) ) K(x - y)  e^{- i \frac{\delta}{r} v_k \cdot y} (V_Q (U)) ^{-1}  \chi_{Q'} (y) \, dy \right) \right\|
 \\ & \lesssim  \sum_k  \sum_{j = 1}^n |a_k| \left\|  (W^\frac{1}{p} [T, B] (g_k \vec{e}_j)) (x)  \right\| \end{align*} where \begin{equation*} g_k (y) =   e^{ -i \frac{\delta}{r} v_k \cdot y } (V_Q (U)) ^{-1}  \chi_{Q'} (y) \end{equation*}  and where the second inequality follows from the fact that $\| S_Q(x)  e^{ i \frac{\delta}{r} v_k \cdot x }\| \leq 1$ for a.e. $x \in \Rd$.

But as $|x_0 - y_0 | = r \delta^{-1} z_0$, we can pick some $C > 1$ only depending on $K$ where  $\tilde{Q} = Q(x_0, C r) $ satisfies $Q \cup Q' \subseteq \tilde{Q}$.  Combining this with the previous estimates, we have from the absolute summability of the $a_n's $ and the boundedness of $[T, B]$ from $L^p(U)$ to $L^p(W)$ that

$\displaystyle \begin{aligned} & \left(\int_{Q} \|  W^{\frac{1}{p}} (x) (B(x) - m_{Q'} B ) (V_Q (U)) ^{-1}\| ^p \, dx \right) ^{\frac{1}{p}} \nonumber \\ &   \leq \sum_k \sum_{j = 1}^n |a_k| \|  (W^\frac{1}{p} (x) [T, B] (g_k \vec{e}_j))     \|_{L^p}  \\ & \lesssim   \sup_k  \sum_{j = 1}^n \| U^{\frac{1}{p}} g_k \vec{e_j} \|_{L^p } \\ & \leq \sum_{j = 1}^n \|\chi_{Q'}  U^{\frac{1}{p}} (V_Q (U)) ^{-1} \vec{e}_j\|_{L^p} \\ & \lesssim   |Q|^\frac{1}{p}  \end{aligned} $

\noindent since the A${}_p$ condition gives us that

 $\displaystyle \begin{aligned} \sum_{j = 1}^n \||Q|^{-\frac{1}{p}}\chi_{Q'}  U^{\frac{1}{p}} (V_Q (U)) ^{-1} \V{e}_j \|_{L^p} & \lesssim   \sum_{j = 1}^n \||\tilde{Q}|^{-\frac{1}{p}} \chi_{\tilde{Q}} (V_Q (U)) ^{-1} U^{\frac{1}{p}}  \V{e}_j \|_{L^p}  \\  &  \lesssim \sum_{j = 1}^n \| |\tilde{Q}|^{-\frac{1}{p}} \chi_{\tilde{Q}} V_{\tilde{Q}} ' (U)  U^{\frac{1}{p}}  \vec{e}_j \ \|_{L^p} \\ & \lesssim \|V_{\tilde{Q}} (U)  V_{\tilde{Q}} ' (U) \|  \lesssim \|U\|_{\text{A}_p} ^\frac{1}{p}.\end{aligned}$ \\

Finally, we can use a simple argument from \cite{IM} to get

  $\displaystyle \begin{aligned} & \left(\int_{Q} \|  W^{\frac{1}{p}} (x) (B(x) - m_{Q} B ) (V_Q (U)) ^{-1}\| ^p \, dx \right) ^{\frac{1}{p}}  \\    & \leq  \left(\int_{Q} \|  W^{\frac{1}{p}} (x)  (B(x) - m_{Q'} B ) (V_Q (U)) ^{-1}\| ^p \, dx \right) ^{\frac{1}{p}} \\ & + \left(\int_{Q} \|  W^{\frac{1}{p}} (x)  (m_{Q} B - m_{Q'} B ) (V_Q (U)) ^{-1}\| ^p \, dx \right) ^{\frac{1}{p}}\end{aligned} $

    \noindent and

     $\displaystyle \begin{aligned} & \left(\int_{Q} \|  W^{\frac{1}{p}} (x)  (m_{Q} B - m_{Q'} B ) (V_Q (U)) ^{-1}\| ^p \, dx \right) ^{\frac{1}{p}} \\ & =  \left(\frac{1}{|Q|} \int_{Q} \left\| \frac{1}{|Q|} \int_{Q}  W^{\frac{1}{p}} (x)  ( B(y) - m_{Q'} B ) (V_Q (U)) ^{-1}  \, dy \right\| ^{p} \, dx \right) ^\frac{1}{p} \\ & \leq \left(\frac{1}{|Q|} \int_{Q} \left( \frac{1}{|Q|} \int_{Q} \| W^{\frac{1}{p}} (x) W^{-\frac{1}{p}} (y)\| \, \right.  \right.\\ & \qquad \times \left. \left. \|  W^{\frac{1}{p}} (y) ( B(y) - m_{Q'} B )(V_Q (U)) ^{-1} \|  \, dy \right) ^{p} \, dx \right) ^\frac{1}{p}     \\ & \leq \left(\frac{1}{|Q|} \int_{Q} \left(\frac{1}{|Q|} \int_{Q}  \| W^{\frac{1}{p}} (x) W^{-\frac{1}{p}} (y) \|^{p'} \, dy \right)^\frac{p}{p'} \, dx \right)^\frac{1}{p}  \\ & \qquad \times  \left(\frac{1}{|Q|} \int_Q  \| W^{\frac{1}{p}}(y)(B(y) - m_{Q'} B )(V_Q (U)) ^{-1} \|^{p} \, dy \right) ^\frac{1}{p} \\ & \leq \|W\|_{\text{A}_p}  \left(\frac{1}{|Q|} \int_Q  \| W^{\frac{1}{p}}(y)(B(y) - m_{Q'} B ) (V_Q (U)) ^{-1} \|^{p} \, dy \right) ^\frac{1}{p}. \qedhere
\end{aligned} $
\end{proof}

\section{H${}^1$-BMO duality when $p = 2$}

In this section we will prove Theorem \ref{H1BMOThm}.  Note that a similar unweighted matrix result was proved in \cite{BW}, and like the proof in \cite{BW}, our proof will also be a matrix extension of the proof in \cite{LLL} with the major difference being that our proof will solely utilize (b) in Theorem \ref{CarEmbedThm} (rather than in \cite{LLL} where condition (c) in the scalar setting is used.)   Furthermore, while we are only interested in the sequence space defined by $\BMOWUtwoD$, it is clear that our proof can be modified to provide a genuine matrix weighted version of the H${}^1$-BMO duality result in \cite{BW}.\\

\noindent \textit{Proof of Theorem \ref{H1BMOThm}} :
Note that for convenience we will write $S_{W^{-1}}$ for $S_{W^{-1}, \D}$.  Also note that throughout the proof we will track the A${}_2$ characteristic dependency on $W$ and $U$, and in particular write ``$A \lesssim B$" to denote that $A \leq C B$ for some unimportant constant $C$ that is independent of $W$ and $U$.  First we prove that every $B \in \BMOWUtwoD$ defines a bounded linear functional on $H^1_{W, U, \D}$. To that end,

  $\displaystyle \begin{aligned} |\ip{\Phi}{B}_{L^2}| & \leq \sum_{\varepsilon \in \S} \sum_{I \in \D} |\tr \{\Phi_I ^\varepsilon (B_I ^\varepsilon) ^*\}| \nonumber \\ & = \sum_{\varepsilon \in \S} \sum_{I \in \D} |\tr \{(m_I W)^{-\frac12} (M_U \Phi)_I  ^\varepsilon ((m_I W)^{\frac12} B_I ^\varepsilon (m_I U)^{-\frac12} )^*\}| \nonumber  \\ & \leq \sum_{\varepsilon \in \S} \sum_{I \in \D} \|(m_I W)^{-\frac12} (M_U \Phi)_I  ^\varepsilon \| \|(m_I W)^{\frac12} B_I ^\varepsilon (m_I U)^{-\frac12} \|. \nonumber \end{aligned}$ \\

  As before let $M$ be the unweighted Hardy-Littlewood maximal function and define the sets $\Omega_k, \W{\Omega}_k,$ and $B_k$ by

   $\displaystyle \begin{aligned} & \Omega_k  = \{x \in \Rd : S_{W^{-1}} (M_U \Phi) (x) > 2^k\}, \\ &  B_k = \{I \in \D : |I \cap \Omega_k| > \frac12 |I| \text{ and } |I \cap \Omega_{k +1}| \leq \frac12 |I| \},  \\ & \W{\Omega}_k = \{ x \in \Rd : M(1_{\Omega_k}) > \frac12 \}.  \end{aligned}$ \\

Clearly $\Omega_k \subseteq \W{\Omega} _k$.  Furthermore if $x \in I$ and $I \in B_k$ then  $ M(1_{\Omega_k}) (x) > \frac{|I \cap \Omega_k|}{|I|} > \frac12 $ so that $I \in B_k$ implies that $I \subseteq \W{\Omega}_k$. Since $S_{W^{-1}} M_U \Phi \in L^1$ we have that $I \in B_k$ for some $k \in \Z$ if $\Phi_I ^\varepsilon \neq 0$ for some $\varepsilon \in \S$.  In particular, since $U$ and $W$ are positive definite a.e., we have that $ S_{W^{-1}} (M_U \Phi) (x) > 0$ when $\Phi_I ^\varepsilon \neq 0$, which combined with the fact that $S_{W^{-1}} (M_U \Phi)  \in L^1$ easily implies the claim.

Thus, if $\W{I}$ denotes the collection of maximal $I \in B_k$ then we have by maximality and two uses of the Cauchy-Schwarz inequality\begin{align} |\ip{\Phi}{B}_{L^2}| & \leq \sum_{\varepsilon \in \S} \sum_{k \in \Z} \sum_{\W{I} \in B_k}  \sum_{\substack{I \subseteq \W{I} \\ I \in B_k}} \|(m_I W)^{-\frac12} (M_U \Phi)_I  ^\varepsilon \| \|(m_I W)^{\frac12} B_I ^\varepsilon (m_I U)^{-\frac12} \| \nonumber \\ & \leq  \sum_{\varepsilon \in \S} \sum_{k \in \Z} \sum_{\W{I} \in B_k} \left(\sum_{\substack{I \subseteq \W{I} \\ I \in B_k}} \|(m_I W)^{-\frac12} (M_U \Phi)_I  ^\varepsilon \|^2\right)^\frac12 \\ & \qquad \times   \left(\sum_{\substack{I \subseteq \W{I} \\ I \in B_k}} \|(m_I W)^{\frac12} B_I ^\varepsilon (m_I U)^{-\frac12} \|^2 \right)^\frac12 \nonumber \\ &  \leq  \|B\|_{\BMOWUtwo} \sum_{\varepsilon \in \S} \sum_{k \in \Z} \sum_{\W{I} \in B_k} |\W{I}|^\frac12  \left(\sum_{\substack{I \subseteq \W{I} \\ I \in B_k}} \|(m_I W)^{-\frac12} (M_U \Phi)_I  ^\varepsilon \|^2\right)^\frac12 \nonumber \\ & \leq \|B\|_{\BMOWUtwo} \sum_{\varepsilon \in \S} \sum_{k \in \Z}  |\W{\Omega}_k| ^\frac12 \left(\sum_{ I \in B_k} \|(m_I W)^{-\frac12} (M_U \Phi)_I  ^\varepsilon \|^2\right)^\frac12. \label{estPhi}\end{align}

We now show that \begin{equation} \label{estBk} \sum_{\varepsilon \in \S} \sum_{ I \in B_k} \|(m_I W)^{-\frac12} (M_U \Phi)_I  ^\varepsilon \|^2 \lesssim \Atwo{W} 2^{2k }|\W{\Omega}_k| \end{equation} where the implied constant is independent of $W$.  To that end, we have \begin{equation*} \int_{\W{\Omega}_k \backslash \Omega_{k + 1}} \left(S_{W^{-1}} (M_U \Phi) (x)\right)^2 \, dx \leq 2^{2k + 2} |\W{\Omega}_k \backslash \Omega_{k + 1}| \leq 2^{2k + 2} |\W{\Omega}_k | \end{equation*} while if $\{\V{e}_j\}_{j = 1}^n$ is any orthonormal basis of $\C$ and we define \begin{equation*} W_I (x) = (m_I W^{-1} )^{-\frac12} W^{-1} (x) (m_I W^{-1})^{-\frac12} \end{equation*} then
\begin{align} &\int_{\W{\Omega}_k \backslash \Omega_{k + 1}}   \left(S_{W^{-1}} (M_U \Phi) (x)\right)^2 \, dx \nonumber \\ & \gtrsim \sum_{j = 1}^n \int_{\W{\Omega}_k \backslash \Omega_{k + 1}} \sum_{\varepsilon \in \S} \sum_{I \in B_k} \frac{|W^{-\frac12} (x) (M_U \Phi)_I ^\varepsilon \vec{e}_j |^2}{|I|} 1_I(x)  \, dx  \nonumber \\ & = \sum_{j = 1}^n \int_{\W{\Omega}_k \backslash \Omega_{k + 1}} \sum_{\varepsilon \in \S} \sum_{I \in B_k} \frac{1}{|I|}\left\langle W_I(x)  (m_I W^{-1})^\frac12 (M_U \Phi)_I ^\varepsilon \vec{e}_j, \right. \nonumber  \\ & \qquad  \left. (m_I W^{-1})^\frac12 (M_U \Phi)_I ^\varepsilon \V{e}_j \right\rangle_{\C} 1_I(x)  \, dx \nonumber \\ & =  \sum_{j = 1}^n  \sum_{\varepsilon \in \S}  \sum_{I \in B_k} \int_{I \cap(\W{\Omega}_k \backslash \Omega_{k + 1})} \frac{1}{|I|} \left\langle W_I(x)  (m_I W^{-1})^\frac12 (M_U \Phi)_I ^\varepsilon \vec{e}_j,  \right. \nonumber \\ &  \qquad  \left. (m_I W^{-1})^\frac12 (M_U \Phi)_I ^\varepsilon \V{e}_j \right\rangle_{\C}   \, dx  \nonumber \\ &  = \sum_{j = 1}^n  \sum_{\varepsilon \in \S} \sum_{I \in B_k} \int_{I \cap(\W{\Omega}_k \backslash \Omega_{k + 1})} \frac{1}{|I|}{|W_I ^\frac12 (x)  (m_I W^{-1})^\frac12 (M_U \Phi)_I ^\varepsilon \vec{e}_j|^2}   \, dx \nonumber \\ & =  \sum_{j = 1}^n  \sum_{\varepsilon \in \S} \sum_{I \in B_k}  \frac{|(m_I W^{-1})^\frac12 (M_U \Phi)_I ^\varepsilon \vec{e}_j|^2}{|I|}   \int_{I \cap(\W{\Omega}_k \backslash \Omega_{k + 1})} |W_I ^\frac12 (x) \V{e}_{I, j}|^2   \, dx \label{H1BMOProofEst} \end{align}  where  \begin{equation*} \V{e}_{I, j} = \left\{
     \begin{array}{lr}
       \frac{(m_I W^{-1})^\frac12 (M_U \Phi)_I ^\varepsilon \vec{e}_j}{|(m_I W^{-1})^\frac12 (M_U \Phi)_I ^\varepsilon \vec{e}_j|} & \text{ if } (m_I W^{-1})^\frac12 (M_U \Phi)_I ^\varepsilon \vec{e}_{I, j}  \neq 0\\
       0 & \text{ if } (m_I W^{-1})^\frac12 (M_U \Phi)_I ^\varepsilon \vec{e}_{I, j}  = 0.
     \end{array}
   \right.
\end{equation*}

However, since $I \subseteq B_K$ we have \begin{align*} \eqref{H1BMOProofEst} & = \sum_{j = 1}^n  \sum_{\varepsilon \in \S}  \sum_{I \in B_k}  \frac{|(m_I W^{-1})^\frac12 (M_U \Phi)_I ^\varepsilon \vec{e}_j|^2}{|I|}   \int_{I  \backslash \Omega_{k + 1}} |W_I ^\frac12 (x) \V{e}_{I, j}|^2   \, dx \\ & \geq \sum_{j = 1}^n  \sum_{\varepsilon \in \S} \sum_{I \in B_k}  \frac{|(m_I W)^{-\frac12} (M_U \Phi)_I ^\varepsilon \vec{e}_j|^2}{|I|}   \int_{I  \backslash \Omega_{k + 1}} |W_I ^\frac12 (x) \V{e}_{I, j}|^2   \, dx.\end{align*}

   Now by Lemma $3.5$ in \cite{TV}, we have that $W_I$ for each $I \in \D$ is a matrix A${}_2$ weight with the same A${}_2$ characteristic as that of $W$.  Furthermore, since each of the nonzero $\V{e}_{I, j}$ are unit vectors, each $|W_I ^\frac12 (x) \V{e}_{I, j}|^2$ is a scalar A${}_2$ weight with A${}_2$ characteristic no greater than that of $W$ (see the proof of Lemma $3.6$ in \cite{TV}). Thus, since \begin{equation*} |I \backslash \Omega_{k + 1}| \geq \frac12 |I| \end{equation*} we have by standard arguments in the theory of (scalar) weighted norm inequalities that \begin{equation*} \frac{|W_I ^\frac12  \V{e}_{I, j}|^2 (I \backslash \Omega_{k + 1})}{|W_I ^\frac12  \V{e}_{I, j}|^2(I)} \geq \frac{1}{\Atwo{W}} \left(\frac{|I \backslash \Omega_{k + 1}|}{|I|}\right)^2 \geq \frac{1}{4 \Atwo{W}}. \end{equation*}

   Furthermore,  \begin{equation*} \frac{|W_I ^\frac12  \V{e}_{I, j}|^2(I)}{|I|} = \frac{1}{|I|} \int_I \ip{(m_I W^{-1})^{-\frac12} W^{-1}(x) (m_I W^{-1})^{-\frac12} \V{e}_{I, j} }{\V{e}_{I, j}}_{\C} \, dx  = 1 \end{equation*} for each nonzero $\V{e}_{I, j}$, which clearly proves \eqref{estBk}.

   Finally combining \eqref{estPhi} with \eqref{estBk} and using the standard $L^{1, \infty}$ maximal function boundedness, we have \begin{align*}  |\ip{\Phi}{B}_{L^2}| & \lesssim  \Atwo{W} ^\frac12 \|B\|_{\BMOWUtwo}  \sum_{k \in \Z}  |\W{\Omega}_k|  2^{k } \\ & \lesssim \Atwo{W} ^\frac12 \|B\|_{\BMOWUtwo}  \sum_{k \in \Z}  |\Omega_k|  2^{k } \\ & \leq \Atwo{W} ^\frac12 \|B\|_{\BMOWUtwo}\|S_{W^{-1}} (M_U \Phi)\|_{L^1}. \end{align*}

Conversely let $\ell \in (H^1 _{W, U})^*$  and let $\{E_j\}_{j = 1}^{n^2}$ be the standard orthonormal basis of $n \times n$ matricies under the inner product $\ip{}{}_{\tr}$.  Clearly if $\Phi \in H^1 _{W, U}$ then the Haar expansion of $\Phi$ converges to $\Phi$ in $H^1 _{W, U}$ so by continuity and linearity we have \begin{equation*} \ell(\Phi)   = \sum_{j = 1}^{n^2} \sum_{\varepsilon \in \S} \sum_{I \in \D}  \ip{\Phi_I ^\varepsilon}{E_j}_{\tr} \ell(E_j h_I ^\varepsilon)  = \ip{\Phi}{B}_{L^2}  \end{equation*} where \begin{equation*} B = \sum_{j = 1}^{n^2} \sum_{\varepsilon \in \S} \sum_{I \in \D} \overline{\ell(E_j h_I ^\varepsilon)} h_I ^\varepsilon E_j ^*   \end{equation*} so that the proof will be complete if we can show that $B \in \BMOWUtwo$.

To that end, by duality we have \begin{align*} & \left(\frac{1}{|J|}  \sum_{\varepsilon \in \S}\sum_{I \in \D(J)} \|(m_I W)^{\frac12} B_I ^\varepsilon (m_I U)^{-\frac12} \|^2\right)^\frac12  \\ & = \left(\frac{1}{|J|} \sum_{\varepsilon \in \S}\sum_{I \in \D(J)} \| (m_I U)^{-\frac12} (B_I ^\varepsilon)^* (m_I W)^{\frac12} \|^2 \right)^\frac12 \\ & \leq \frac{1}{|J|^\frac12} \sup_{\|\{S_I ^\varepsilon\}\|_{\ell^2} = 1} \left|\sum_{\varepsilon \in \S} \sum_{I \in \D(J)} \ip{(m_I U)^{-\frac12} (B_I ^\varepsilon)^* (m_I W)^{\frac12}}{S_I ^\varepsilon}_{\tr}\right| \\ & \leq \frac{1}{|J|^\frac12} \sup_{\|\{S_I ^\varepsilon\}\|_{\ell^2} = 1}  \left|\sum_{\varepsilon \in \S} \sum_{I \in \D(J)} \ip{(m_I W)^{\frac12} (S_I ^\varepsilon)^* (m_I U)^{-\frac12}}{B_I ^\varepsilon}_{\tr}\right| \\ & = \frac{1}{|J|^\frac12} \sup_{\|\{S_I ^\varepsilon\}\|_{\ell^2} = 1} |\ip{S_{J, W, U} ^*}{B}_{L^2}| \\ & \leq  \sup_{\|\{S_I ^\varepsilon\}\|_{\ell^2} = 1} \frac{1}{|J|^\frac12} \|\ell\|  \|S_{J, W, U} ^*\|_{H^1 _{W, U}} \end{align*} where \begin{equation*}S_{J, W, U} =  \sum_{\varepsilon \in \S} \sum_{I \in \D(J)}(m_I U)^{-\frac12} S_I ^\varepsilon (m_I W)^{\frac12} h_I ^\varepsilon .\end{equation*}   The proof will then be completed (and the interchanging of $\ell$ and summation will be justified) if we can show that $\|S_{J, W, U} ^*\|_{H^1 _{W, U}} \lesssim  |J|^\frac12.$

However, by  Cauchy-Schwarz, \begin{align*}  \|S_{J, W, U} ^*\|_{H^1 _{W, U}} & = \int_J  \left(\sum_{\varepsilon \in \S} \sum_{I \in \D(J)} \frac{\|W^{-\frac{1}{2}} (x) (m_I W)^{\frac12} (S_I^\varepsilon)^* \|^2  }{|I|} 1_I (x) \right)^\frac12 \, dx \\ & \leq |J|^\frac12 \left(\sum_{\varepsilon \in \S} \sum_{I \in \D(J)} \frac{1}{|I|} \int_I \|W^{-\frac{1}{2}} (x) (m_I W)^{\frac12} (S_I^\varepsilon)^* \|^2  \, dx \right)^\frac12 \\ & \leq \Atwo{W} ^\frac12 |J|^\frac12 \left(\sum_{\varepsilon \in \S} \sum_{I \in \D(J)} \|S_I^\varepsilon\| ^2 \right)^\frac12  \leq \Atwo{W} ^\frac12 |J|^\frac12 \end{align*} since $\|\{S_I ^\varepsilon\}\|_{\ell^2} = 1$. \hfill $\square.$

\section{Completion of the proofs}  In this section we will complete the proof of Theorems \ref{CommutatorThm}  and \ref{MatrixWeightThm}.  First, note that it is a by now standard fact that for any cube $Q$ there exists $1\leq t\leq 2^d $ and $Q_t\in \D^t$ such that $Q\subset Q_t$ and $\ell(Q_t)\leq 6\ell(Q)$ where $\ell(Q)$ is the side length of $Q$ and $\D^t = \{2^{-k} ([0, 1) ^d + m  + (-1)^k t)  : k \in \Z, m \in \Z^d\}$.  As was mentioned before, Theorem \ref{CommutatorThm} will be completed by the following.
 \begin{lemma} \label{JNlemma} If $1 < p < \infty, \D$ is a dyadic grid  and $W, U$  are matrix A${}_p$ weights, then we have $\BMOWUD = \BMOWUAltD$.  Furthermore we have that \begin{equation*} \bigcup_{t = 1}^{2^d} \BMOWUDt = \BMOWU \text{ and } \bigcup_{t = 1}^{2^d} \BMOWUAltDt = \BMOWUAlt \end{equation*}

 \end{lemma}

\begin{proof} Let $B \in \BMOWUD$ so for some $\epsilon > 1$ (which by H\"{o}lder's inequality we assume is in the interval $(0, 1)$) we have by dyadic Littlewood-Paley theory that \begin{align} \sup_{\substack{I \subset \R^d \\  I \text{ is a cube}}} & \frac{1}{|I|} \int_{I} \left( \sum_{\varepsilon \in \S} \sum_{J \in \D(I)} \frac{\| (m_I W ^\frac{1}{p}) B_J ^\varepsilon (m_I U^{\frac{1}{p}})^{-1} \| ^2}{|J|} \chi_J (x)  \right) ^\frac{1 + \epsilon}{2} \, dx \nonumber \\ & \approx \sup_{\substack{I \subset \R^d \\  I \text{ is a cube}}}  \frac{1}{|I|} \int_{I} \left( \sum_{\varepsilon \in \S} \sum_{J \in \D(I)} \frac{\| V_I (W) B_J ^\varepsilon (V_I (U))^{-1} \| ^2}{|J|} \chi_J  (x)  \right) ^\frac{1 + \epsilon}{2} \, dx \nonumber \\ & < \infty \nonumber   \label{BAssump} \end{align} where we have used Lemma \ref{RedOp-AveLem} twice.   However, we have that $B \in \BMOWUAltD$ if and only if \begin{equation*} \sup_{\substack{I \subset \R^d \\  I \text{ is a cube}}} \, \frac{1}{|I|} \sum_{\varepsilon \in \S} \sum_{J \in \D(I)}  \|V_J (W) B_J ^\varepsilon (V_J (U)) ^{-1}\|^2 < \infty \end{equation*}  which by Theorem 3.1 in \cite{NTV} is equivalent to \begin{equation}  \label{Btoprove} \sup_{\substack{I \subset \R^d \\  I \text{ is a cube}}}  \frac{1}{|I|} \int_{I} \left( \sum_{\varepsilon \in \S}  \sum_{J \in \D(I)} \frac{\|V_J (W) B_J ^\varepsilon (V_J (U)) ^{-1}\|^2 }{|J|} \chi_J  (x)  \right) ^\frac{1 + \epsilon}{2} \, dx < \infty. \end{equation} Using the stopping time notation from Section $2$, note that  $J \in \F(K)$ implies that $\|V_J (W) (V_K (W))^{-1}\| \lesssim 1$ and $\|V_K (U) (V_J(U))^{-1}\| \lesssim 1$, so that for fixed $I \in \D$,

\begin{align}
&  \frac{1}{|I|} \int_{I} \left( \sum_{\varepsilon \in \S} \sum_{J \in \D(I)} \frac{\|V_J (W) B_J ^\varepsilon (V_J(U)) ^{-1} \| ^2}{|J|} \chi_J  (x)  \right) ^\frac{1 + \epsilon}{2} \, dx \nonumber \\ & =
\frac{1}{|I|} \int_I \left( \sum_{j = 1} ^\infty \sum_{K \in \J^{j - 1} (I)} \sum_{\varepsilon \in \S} \sum_{J \in \F(K)}  \frac{\|V_J (W) B_J ^\varepsilon (V_J (U)) ^{-1} \|^2 }{|J|} \chi_J  (x) \right) ^\frac{1 + \epsilon}{2} \, dx \nonumber \\ & \lesssim \frac{1}{|I|}  \int_I \left(  \sum_{j = 1} ^\infty \sum_{K \in \J^{j - 1} (I)}  \sum_{\varepsilon \in \S} \sum_{J \in \D(K)} \frac{\|V_K(W) B_J ^\varepsilon  (V_K (U)) ^{-1} \|^2 }{|J|} \chi_J  (x) \right) ^ \frac{1 + \epsilon}{2} \, dx \nonumber \\ & \leq  \frac{1}{|I|}    \sum_{j = 1} ^\infty \sum_{K \in \J^{j - 1} (I)} \int_K \left( \sum_{\varepsilon \in \S} \sum_{J \in \D(K)} \frac{\|V_K (W) B_J ^\varepsilon (V_K (U)) ^{-1} \|^2 }{|J|} \chi_J  (x) \right) ^\frac{1 + \epsilon}{2} \, dx \nonumber \\ & \lesssim \frac{1}{|I|}   \sum_{j = 1} ^\infty \sum_{K \in \J^{j - 1} (I)} |K|   \lesssim \sum_{j = 1}^\infty 2^{-(j - 1)} < \infty. \nonumber \end{align}

Conversely, for $\epsilon > 0$ small enough we have \begin{align*} & \frac{1}{|I|} \int_I \|V_I (W)(B(x) - m_I B) (V_I (U))^{-1}\|^{1 + \epsilon} \, dx \\ & \leq \frac{1}{|I|} \int_I \|V_I (W) W^{-\frac{1}{p}} (x)\| ^{1 + \epsilon}
\|W^{\frac{1}{p}}(x) (B(x) - m_I B) (V_I (U))^{-1}\|^{1 + \epsilon} \, dx \\ & \leq \left(\frac{1}{|I|} \int_I \|V_I (W) W^{-\frac{1}{p}} (x)\|^{\frac{p(1 + \epsilon)}{p - 1 - \epsilon}} \, dx \right)^{\frac{p - 1 - \epsilon}{p}} \\ & \qquad \times \left( \frac{1}{|I|} \int_I
\|W^{\frac{1}{p}}(x) (B(x) - m_I B) (V_I (U))^{-1}\|^{p} \, dx\right)^\frac{1 + \epsilon}{p} \\ & \lesssim \left( \frac{1}{|I|} \int_I \|W^{\frac{1}{p}}(x) (B(x) - m_I B) (V_I (U))^{-1}\|^{p} \, dx\right)^\frac{1 + \epsilon}{p}  \end{align*} by the reverse H\"{o}lder inequality.

As for the last two statements, one can argue as we did towards the end of the proof of Theorem \ref{RieszTransNec} and we will leave these simple details to the interested reader.  \end{proof}

We now prove Theorem \ref{MatrixWeightThm}\\

 \noindent \textit{Proof of Theorem \ref{MatrixWeightThm}:}

 Let $\Lambda$ be a matrix A${}_p$ weight and let $R$ be any of the Riesz transforms.  If $W = (U^* \Lambda^\frac{2}{p} U)^\frac{p}{2}$, then \begin{align*} \| R U\vec{f}\|_{L^p(\Lambda)}   & \lesssim  \|  U\vec{f}\|_{L^p(\Lambda)}    \left(\inrd |\Lambda ^\frac{1}{p} U \vec{f}|^p \, dx \right)^\frac{1}{p} = \left(\inrd \ip{U^*\Lambda ^\frac{2}{p} U \vec{f}}{ \vec{f}}_{\C} ^\frac{p}{2} \, dx \right)^\frac{1}{p}
   \\ &  = \left(\inrd |[ (U^* \Lambda ^\frac{2}{p} U)^\frac{p}{2}]^\frac{1}{p} \vec{f}| ^p \, dx \right)^\frac{1}{p}  = \|\vec{f}\|_{L^p(W)}. \end{align*}

 On the other hand, the easy computation above and the fact that $W$ is a matrix A${}_p$ weight gives us that \begin{align*}  \|  U R \vec{f}\|_{L^p(\Lambda)}  = \|R\vec{f}\|_{L^p(W)}  \lesssim \|\vec{f}\|_{L^p(W)}. \end{align*}

 \hfill $\square$.

As was mentioned in the introduction, it is rather curious to examine the very special case of $p = 2, U = W,$ and $\Lambda = W^{-1}$ where $W$ is a matrix A${}_2$ weight, which gives us that $W \in \text{BMO}_{W^{-1}, W}$ if $W$ is a matrix A${}_2$ weight.  Thanks to Theorem \ref{CarEmbedThm}, this in conjunction with some elementary linear algebra and the matrix A${}_2$ condition proves the following result.

\begin{proposition} \label{SummationProp}  If $W$ is a matrix A${}_2$ weight and $\D$ is a dyadic grid, then $W$ satisfies  \begin{equation} \label{FKP} \sup_{J \in \D} \frac{1}{|J|} \sum_{\varepsilon \in \S} \sum_{I \in \D(J)} \|(m_I W)^{-\frac12 } W_I ^\varepsilon (m_I W)^{-\frac12} \|^2 < \infty, \end{equation}  \begin{equation} \label{Buckley} \sum_{\varepsilon \in \S} \sum_{I \in \D(J)} W_I ^\varepsilon (m_I W)^{-1} W_I ^\varepsilon < C |J| m_J W \end{equation} for some $C$ independent of $J$, and \begin{equation} \label{IsralSummationProp}  \frac{1}{|J|}\sum_{\varepsilon \in \S} \sum_{I \in \D(J)}  m_I (W^{-1} ) W_I^\varepsilon m_I (W^{-1}) W_I ^\varepsilon m_I (W ^{-1})  < C m_J (W^{-1}) \end{equation} for some $C$ independent of $J$.

 \end{proposition}

As was mentioned in the introduction, while \eqref{FKP} in the scalar setting is known as the Fefferman-Kenig-Pipher inequality and is known in the matrix setting (see \cite{TV} when $d = 1$ and \cite{CW} when $d > 1$), inequality \eqref{Buckley} is to the author's knowledge new (and in the scalar setting is well known as Buckley's inequality, see \cite{Buck}).  Note that the interest in these two inequalities stems from their use in sharp matrix weighted norm inequalities.  In particular, it was shown in \cite{TV, CW} that the supremum in \eqref{FKP} is comparable to $\log ( 1 + \Atwo{W})$, which in \cite{BPW, CW} is used to prove quantitative matrix weighted square function bounds.  Note that these square function bounds immediately give quantitative matrix weighted norm inequalities for Riesz transforms, and in particular give that \begin{equation} \label{BPWbound} \|R\|_{L^2(W) \rightarrow L^2(W)} \lesssim \Atwo{W} ^\frac32 \log (1 + \Atwo{W}) \end{equation} for any of the Riesz transforms $R$.

Furthermore, as was pointed out in \cite{BPW}, if one could prove that $C \approx \Atwo{W}^2$ in \eqref{Buckley} when $d = 1$ (which is known to be sharp in the scalar setting), then one would be able to improve the right hand side of \eqref{BPWbound} for the Hilbert transform to $\Atwo{W}^\frac32$ (and while not stated in \cite{BPW}, a similar statement can be said for any of the Riesz transforms). While this appears to be quite challenging, we can at least prove the following.   \begin{proposition} \label{BuckleyProp} The constant $C$ in \eqref{Buckley} can be picked to satisfy $C \approx \Atwo{W}^2 \log ( 1 + \Atwo{W}).$ \end{proposition}

\begin{proof} The proof is similar to the cases $(b) \Rightarrow (a)$ and $(a) \Rightarrow (c)$ in Theorem \ref{CarEmbedThm} but is simpler.  More precisely, as before let \begin{equation*} M_W ' \vec{f} (x) = \sup_{\D \ni I \ni x} m_I |  (m_I W)^{\frac12} W^{-\frac12} \vec{f}|. \end{equation*}  Then by the above mentioned bound for the supremum in \eqref{FKP} and the standard unweighted dyadic Carleson embedding theorem, we have for $\vec{f} \in L^2$
\begin{align*}   \sum_{\varepsilon \in \S} \sum_{I \in \D} & | (m_I W)^{-\frac12} W_I ^\varepsilon m_I (W^{-\frac12 \vec{f}})|^2
\\ & \leq  \sum_{\varepsilon \in \S} \sum_{I \in \D} \| (m_I W)^{-\frac12} W_I ^\varepsilon  (m_I W)^{-\frac12} \|^2   (m_I |  (m_I W)^{\frac12} W^{-\frac12} \vec{f}| )^2
\\ &  \leq  \sum_{\varepsilon \in \S} \sum_{I \in \D} \| (m_I W)^{-\frac12} W_I ^\varepsilon  (m_I W)^{-\frac12} \|^2   (m_I (M_W ' \vec{f}))^2
\\ & \lesssim \log ( 1 + \Atwo{W}) \| M_W ' \vec{f}\|_{L^2} ^2
\\ & \lesssim \|W\|_{\text{A}_2}^2 \log ( 1 + \Atwo{W}) \|\vec{f}\|_{L^2} ^2 \end{align*}

However, plugging in the testing function $\vec{f} = W^\frac12 \chi_J \vec{e}$ for some $J \in \D$ and $\vec{e} \in \C$ we get that

\begin{align*} \sum_{\varepsilon \in \S} \sum_{I \in \D(J)} | (m_I W)^{-\frac12} W_I ^\varepsilon \vec{e}|^2 & \lesssim \|W\|_{\text{A}_2}^2 \log ( 1 + \Atwo{W})  \|W^\frac12 \chi_J \vec{e} \|_{L^2} ^2
\\ & \lesssim \|W\|_{\text{A}_2}^2 \log ( 1 + \Atwo{W})  |J| | (m_J W)^\frac12 \vec{e} |^2 \end{align*} which is easily seen to be equivalent to \eqref{Buckley} with $C \lesssim \|W\|_{\text{A}_2}^2 \log ( 1 + \Atwo{W}). $ \end{proof}

Let us remark that even defining correct $p\neq 2$ versions of $\eqref{Buckley}$ or $\eqref{FKP}$ looks quite mysterious.  Also,  note that \eqref{IsralSummationProp} appears to be new even in the scalar weighted setting (which is possibly due to the fact that it is not clear where such an inequality can be used.)

We will finish this paper by proving Propositions \ref{p=2JN} and \ref{VectorJN}. \\

\noindent \textit{Proof of Proposition} \ref{p=2JN}:  Since $W$ is a matrix A${}_2$ weight iff $W^{-1}$ is a matrix A${}_2$ weight, the proof follows immediately  by the matrix A${}_2$ condition and Lemma \ref{RedOp-AveLem} in conjunction with the fact that BMO${}_{W^{-1}, W} ^2 = \widetilde{\text{BMO}_{W^{-1}, W} ^2}.$ \hfill $\square$ \\

\noindent To prove Proposition \ref{VectorJN}, we first need to recall that $|W^\frac{1}{p} (x) \V{e}|^p$ is  a scalar A${}_{\infty}$ weight for $\V{e} \in \C$ (with A${}_{\infty}$ constant uniform in $\V{e}$, see \cite{V}) and thus satisfies a reverse H\"{o}lder's inequality. \\

\noindent \textit{Proof of Proposition} \ref{VectorJN}:   We first assume that \eqref{JNVecCond} is true and that $W$ is a matrix A${}_p$ weight.  Then \begin{align*} \frac{1}{|I|} & \int_I |\V{f}(x) - m_I \V{f} | \, dx \\ & \leq  \frac{1}{|I|} \int_I \|V_I(W)  W^{-\frac{1}{p}} (x) \| |W^\frac{1}{p} (x) V_I ^{-1}(W) (\V{f}(x) - m_I \V{f}) | \, dx \\ & \lesssim  \left(\frac{1}{|I|} \int_I  |W^\frac{1}{p} (x) V_I ^{-1}(W) (\V{f}(x) - m_I \V{f}) |^p  \, dx \right)^\frac{1}{p}  \end{align*} by H\"{o}lder's inequality and the matrix A${}_p$ condition.

Conversely, if $W$ is a matrix A${}_{p, \infty}$ weight then there exists $q > p$ where \begin{align*} \left(\frac{1}{|I|}  \int_I \|W^\frac{1}{p} (x) (V_I(W))^{-1} \|^q \right)^\frac{1}{q} & \lesssim \sum_{i = 1}^n \left(\frac{1}{|I|} \int_I |W^\frac{1}{p} (x) (V_I(W))^{-1} \V{e}_i |^q \right)^\frac{1}{q} \\ & \lesssim \sum_{i = 1}^n \left(\frac{1}{|I|} \int_I |W^\frac{1}{p} (x) (V_I(W))^{-1} \V{e}_i |^p \right)^\frac{1}{p} \lesssim 1. \end{align*} Thus, if $\V{f} \in \text{BMO}$ then the classical John-Nirenberg inequality and the following inequality completes the proof:  \begin{align*} \frac{1}{|I|} & \int_I  |W^\frac{1}{p} (x) V_I ^{-1}(W) (\V{f}(x) - m_I \V{f}) |^p  \, dx  \\ & \lesssim \left(\frac{1}{|I|} \int_I |\V{f}(x) - m_I \V{f}| ^{ \frac{q}{q - p}} \, dx \right)^\frac{q-p}{q}. \end{align*}     \hfill $\square$

\end{document}